\documentclass[a4paper,11pt,reqno]{amsart}
\usepackage{amsmath,amssymb}
\usepackage[toc]{appendix}
\usepackage{enumerate}
\usepackage{float}
\usepackage{ifthen}
\usepackage{graphicx}
\usepackage{calc}
\usepackage{tikz}
\usepackage{tikz-qtree}
\usepackage[hypcap]{caption}
\usetikzlibrary{arrows}
\tikzstyle{block}=[draw opacity=0.7,line width=1.4cm]
\nonstopmode\numberwithin{equation}{section}
\usepackage{fancyhdr}
 
\newtheorem{definition}{Definition}[section]
\newtheorem{theorem}{Theorem}[section]

\newtheorem{lemma}{Lemma}[section]
\newtheorem{example}{Example}[section]
\newtheorem{remark}{Remark}[section]

\newtheorem{assum}{Assumption}[section]
\allowdisplaybreaks
{} 
{} 
{} 

 \theoremstyle{plain}

\begin{document}
\title{ Iteratively regularized landweber iteration method: Convergence analysis via  H\"older Stability}
\maketitle
\begin{center}
\center{${\text{Gaurav Mittal}}$ and   $\text{Ankik Kumar Giri}$}
\medskip
{\footnotesize
 \center{ Department of Mathematics, Indian Institute of Technology Roorkee, Roorkee,  India, 247667.}
} \\
\email {gmittal@ma.iitr.ac.in,  {ankik.giri@ma.iitr.ac.in}}
\bigskip
\begin{abstract} In this paper, the local convergence of Iteratively regularized Landweber iteration method is investigated  for solving non-linear  inverse problems in  Banach spaces. Our  analysis mainly relies   on the assumption that the inverse mapping satisfies the H\"older stability estimate locally. We consider both  noisy as well as non-noisy data in our analysis. Under the a-priori choice of stopping index for noisy data, we show that the iterates remain in a certain ball around exact solution and obtain the convergence rates.
 The convergence of the Iteratively  regularized Landweber iterates to the exact solution is shown under certain assumptions in the case of non-noisy data and as a by-product, under  different conditions, two different convergence rates are obtained. 

\hspace{-5mm}\subjclass {}{\textbf{AMS Subject Classifications}: $65$J$15$,  $47$A$52$, $47$J$06$}\\
\hspace{-5mm}
 
\end{abstract}

\end{center}

\section{Introduction}

Let $F: D(F) \subset U \to V: F(u) = v$ be a non-linear forward operator between the Banach spaces $U$ and $V$. The classical meaning of an inverse problem is the determination of $u\in U$, provided $v$ or some approximation of $v$ is given. For further details on inverse problems, see $[16]$ for   Hilbert spaces settings, and $[33]$ for Banach space settings. In general, due to the lack of continuous dependence on the data, almost all the inverse problems are ill-posed in nature. Thus,  regularization methods are needed to find the stable approximate solutions  of the ill-posed  inverse problems.    Variational regularization methods are well known regularization methods for finding the stable approximate solutions and are well studied, see, for instance,  $[16, 26]$.  Nevertheless, iterative methods are often an appealing alternative to variational methods (specifically for large-scale problems). Among all the iterative methods, Landweber iteration method is one of the well known classical methods. For the convergence results of Landweber iteration and its modifications in Hilbert spaces, an  extensive research has been done in $[21, 23]$. In the case of monotone  operators, there is an important role of duality mappings in iterative methods $($see $[5, 8, 9, 36])$. Using the duality mapping, non-linear generalization of the Landweber method  is given in $[10]$ for  Banach spaces. Scherzer, in $[25]$, gave the modification of Landweber iteration method and coined it as  iteratively regularized Landweber iteration method.  This method is highly motivated  from the iteratively regularized Gauss-Newton method introduced by Bakushinskii in $[1]$.  In our study,   the data space $V$ can be any arbitrary Banach space but the model space $U$ needs to be  uniformly convex and smooth (see next section for their formal definitions). In the theory of Banach spaces, Bregman distances play an important role because of their rich geometrical properties and are more convenient to employ rather than Ljapunov functionals to prove the convergence of regularization schemes $[31]$.  And hence, it is more appropriate to derive the convergence rates with the help of Bregman distances. 

Conceptually, convergence rates can be derived with two different approaches for non-linear  problems. First one is on the basis of source and non-linearity conditions, see, for instance, $[16,  26, 27, 33]$ for variational regularization, and $[2, 3, 4, 30, 33]$ for iterative regularization. The second approach relies on the stability estimates which has been derived in $[18]$ for Tikhonov's regularization  method and in $[24]$ for iterative regularization (Landweber iteration method)  in Banach spaces. The results regarding the rates of  convergence using  H\"older stability estimates  and logarithmic stability estimates can also be found in $[12, 15]$ and  $[34, 35]$ respectively.

 In our analysis, we consider the iteratively regularized Landweber iteration scheme which is taken from  $[33]$.  The motivation for this paper comes from $[24]$ in which the convergence rates for Landweber iteration method have been obtained via H\"older stability estimates, however, only non-noisy data is  considered there.  The prime motive  of this work is to study the  convergence of the iterates of Iteratively regularized Landweber iteration method $(2.1)$-$(2.2)$ provided the inverse mapping satisfies  the H\"older stability estimate $(3.2)$  and hence find the convergence rates. Since non-noisy data is taken in $[24]$ for Landweber iteration method,  we want to emphasize that from our results, one can also deduce the convergence rates in the presence of noisy data  for Landweber iteration method.  
   Novelty of this work  is to determine the convergence rates for both the noisy as well as non-noisy data without using the classical approach  based on source conditions as well as the contemporary smoothness concept  known as variational inequalities. 

The plan of this paper is the following:  All the basic results and definitions required in our framework are recapitulated in Section $2$. In the third section, the main result     on the convergence and its rates is stated and proved  in  Theorem $3.1$ along with the necessary  assumptions. In addition, a convergence rate is also established in  Theorem $3.2$  for the special case of  H\"older stability estimates. In Section $4$, we give an  example where our results on the convergence can be applied. At the end, a few conclusions are made.  

\section{Preliminaries}

\begin{definition}{Duality map}: Let $U$ be a Banach space and $U^*$ be its dual space. The  mapping $J_p : U \to 2^{U^*}$ of the convex functional $u \to \frac{1}{p}\|u\|^p$ defined by $$J_p(u) = \{ u^* \in U^*\ | \ \langle u, u^*\rangle = \|u\|\|u^*\|, \|u^*\| = \|u\|^{p-1}\}$$ is known as the duality mapping of $U$ with the gauge function $t \to t^{p-1}$,  where $p>1$. \end{definition}
\begin{example}
Let $a > 1$. Then, for $U = \mathcal{L}^a(\mathbb{R}^n)$ $($the space of measurable functions for which the $a$-th power of the absolute value is Lebesgue integrable$)$, we have $$J_p: \mathcal{L}^a(\mathbb{R}^n) \to \mathcal{L}^b(\mathbb{R}^n)\quad \text{defined by} \quad u(x)\mapsto \|u\|_U^{p-a}|u(x)|^{a-2}u(x),$$ where $a$ and $b$ are conjugate indices. 

\end{example}
In general, $J_p$ is a set valued mapping but we need it to be single-valued in the further analysis. In order to fulfil this condition, we introduce the notions of uniform convexity and uniform smoothness of Banach spaces.
\begin{definition}{Convexity modulus of} $U$: It is a function $\delta: [0,2] \to [0,1]$ defined by

 $$\delta_U(\epsilon) =  \inf \bigg\{\frac{1}{2}\bigg(2-\|u_1+u_2\|\bigg): \ u_1, u_2 \in S, \|u_1-u_2\| \geq \epsilon\bigg\},$$ where $S$ is the boundary of   unit sphere in the Banach space $U$. 
 Further,  if $\delta_U(\epsilon) > 0$ for any $\epsilon \in (0, 2]$, then $U$ is uniformly convex.  

\end{definition}
\begin{definition} {Smoothness modulus of} $U$: It is a function $\rho: [0,\infty) \to [0,\infty)$ defined by

 $$\rho_U(\tau) =  \sup \bigg\{\frac{1}{2}\bigg(\|u_1+\tau u_2\|+ \|u_1-\tau u_2\| - 2\bigg): \ u_1, u_2 \in S\bigg\}, $$ where $S$ is the boundary of   unit sphere in the Banach space $U$. Further,   if $\lim_{\tau \to 0}{\dfrac{\rho_U(\tau)}{\tau}} = 0$, then $U$ is uniformly smooth.  
\end{definition}

\begin{definition} A Banach space $U$ is \begin{enumerate}
\item $p$ convex or convex of power type $p$ if  $\delta_U(\epsilon) \geq Y \epsilon^p$, where $Y>0$ is a constant. 
\item $q$ smooth if   $\rho_U(\tau) \leq Z \tau^q$, where $Z>0$ is a constant. 
\end{enumerate}
\end{definition}
\begin{example}
 The Banach space $U = \mathcal{L}^p(\Sigma)$, where  $p > 1$  and  $\Sigma \subset \mathbb{R}^n$ be an open domain, is uniformly convex as well as uniformly smooth and $$\delta_U(\epsilon) = \begin{cases} \epsilon^2, \qquad 1 < p < 2\\ \epsilon^p, \qquad 2 \leq p < \infty\end{cases} \quad \text{and}\quad \ \  \rho_U(\tau) = \begin{cases} \tau^p, \qquad 1 < p < 2\\ \tau^2, \qquad 2 \leq p < \infty.\end{cases}$$
\end{example}

Next, we  recall the definition of Bregman distance, see $[33, \ \text{Definition}\ 2.56]$. \begin{definition}{Bregman distance}: Let $U$ be a uniformly smooth Banach space and $J_p$ is the duality mapping from $U$ to $U^*$ with the gauge function $t \to t^{p-1}$. Then the functional $$\Delta_p(u_1, u_2) = \frac{1}{p}\|u_1\|^p-\frac{1}{p}\|u_2\|^p-\langle J_p(u_2), u_1-u_2\rangle, \quad  u_1 \in U,$$ is the  Bregman distance of the convex functional $u \to \frac{1}{p}\|u\|^p$ at $u_2 \in U$. \end{definition}The following identity in Lemma $2.1$ is known as three point identity for Bregman distances, for proof see $[33, \ \text{Lemma}\ 2.62]$.
\begin{lemma}
For $u_1, u_2$ and $u_3$ in the Banach space $U$, we have $$\Delta_p(u_1, u_2)= \Delta_p(u_1, u_3)+\Delta_p(u_3, u_2)+\langle J_p(u_3)-J_p(u_2), u_1-u_3\rangle.$$
\end{lemma}
\subsection{Iteratively Regularized Landweber Iteration Method}
In Banach spaces, we consider the following   iteratively regularized Landweber iteration   method given in $[33]$:  \begin{equation} \label{eq1} \begin{split} J_p(u_{k+1}^{\delta} - u_0) = (1-\beta_k) J_p(u_k^{\delta}  - u_0) - \mu F'(u_k^{\delta})^*j_p(F(u_k^{\delta})-v^{\delta}),\end{split}\end{equation} \begin{equation} \label{eq1} \begin{split} \qquad u_{k+1}^{\delta}  = u_0 +  J_q^*(J_p(u_{k+1}^{\delta}-u_0)), \  \text{where}  \quad 0 < \beta_k \leq \beta_{\max} < 1, \ k = 0, 1, 2, \cdots\end{split} \end{equation} Here  $J_p: U \to U^*$, $J_q^*: U^* \to U$, $j_p: V \to V^*$ are duality mappings, $\mu$ is a  positive constant, $u_0= u_0^{\delta}$ is the initial guess of the solution,  $v^{\delta} \in V$ be such that $\|v^{\delta}-v\|\leq \delta$ and $p, q > 1$ are conjugate indices. This iterative scheme is a Gradient type method  resulting from the application of gradient descent  to the misfit $\|F(u)-v\|^p$.
\begin{remark}
For  Hilbert space settings, convergence of  Iteratively regularized Landweber iteration scheme $(2.1)$-$(2.2)$ has been shown in $[25]$  for the noisy data and  the appropriate choice of $\beta_i$'s in $[0, 1]$.  Also convergence rates have been obtained in $[15]$  provided the exact solution satisfies the  source conditions $[16]$. In  $[33, \text{Theorem}\ 7.5]$,  convergence rates have been  obtained for the method $(2.1)$-$(2.2)$ in Banach spaces   by incorporating the following variational inequalities $$|\langle J_p(u^{\dagger}-u_0), u-u^{\dagger}\rangle| \leq \beta \Delta_p^{u_0}(u^{\dagger}, u)^{\frac{1-\nu}{2}}\|F'(u^{\dagger})(u-u^{\dagger})\|^{\nu}, $$ and the non-linearity estimate $$\|(F'(u^{\dagger}+v)-F'(u^{\dagger}))v\| \leq K \|F'(u^{\dagger})v\|^{c_1}\ \Delta_p^{u_0}(u^{\dagger}, v+u^{\dagger})^{c_2},$$ where 
$v \in U$ and  $u, u^{\dagger}+v$ are in some ball of positive radius around the exact solution $u^{\dagger}$, $\nu \in (0, 1]$,  $\beta >0, K > 0$,  $\Delta_p^{u_0}(u^{\dagger}, u ) = \Delta_p(u^{\dagger}-u_0, u-u_0)$ and $c_1$, $c_2$ are properly chosen constants. Here, we study both the convergence and convergence rates by incorporating an alternative condition, namely H\"older type stability $(3.2)$ replacing the variational inequalities and the non-linearity estimate. 
\end{remark}
\begin{remark}
 For solving $F(u)=v$, suppose $v^{\delta}$ is known to us such that $\|v^{\delta}-v\|\leq\delta$ for some $\delta >0$. Then, consider the following iteration scheme:
\begin{equation*} \label{eq1} \begin{split} J_p(u_{k+1}^{\delta} ) =  J_p(u_k^{\delta}) - \mu F'(u_k^{\delta})^*j_p(F(u_k^{\delta})-v^{\delta}) + \beta_k J_p(u_0-u_k^{\delta}),\end{split}\end{equation*} \begin{equation*} \label{eq1} \begin{split} \qquad u_{k+1}^{\delta}  =  J_q^*(J_p(u_{k+1}^{\delta})), \quad  \text{where}  \quad 0 < \beta_k \leq \beta_{\max} < \frac{1}{2}.\end{split} \end{equation*}This is   another version of  Iteratively regularized Landweber Iteration method. In Hilbert spaces,  this method reduces to  the method discussed in $[4]$ with $\mu = 1$.\end{remark}
\begin{remark}
 If $\beta_ k = 0$ for each $k$ in $(2.1)$,  then the resulting method is nothing but the Landweber iteration method  discussed in $[24]$ with  $u_0 = \delta= 0$. 
\end{remark}
Now,  we recall the properties of    duality mappings through which one get to know about the conditions under which the duality mapping $J_p$ is single valued, invertible etc.,  see $[10, 17]$. 
\begin{theorem}
For $p > 1$, the following holds:
\begin{enumerate}
\item For every $u \in U$, the set $J_p(u)$ is non empty.

\item The set $J_p(u)$ is single valued for each $u\in U$ provided the Banach space $U$ is uniformly smooth. 

\item If a  Banach space is uniformly convex    and  uniformly smooth,  then $J_p(u)$ is one-one and onto and its inverse is $J_p^{-1} = J_q^*$, with $J_q^*$ is the duality mapping of $U^*$, where $p, q > 1$ with $\frac{1}{p}+\frac{1}{q} = 1$ and the associated  gauge function is $t \to t^{q-1}$. 

\item Uniform smoothness $($uniform convexity$)$ of a Banach space $U$ is equivalent to the uniform convexity $(\text{uniform smoothness})$ of the dual space $U^*$. 
\end{enumerate}
\end{theorem}

Next result recapitulates the main facts of Bregman distance and its relationship with the norm. See $[33, \ \text{Theorem}\ 2.60]$ for proof of parts $(1)$ and $(4)$ in the following theorem. 

\begin{theorem}
Let $U$ be  a  uniformly convex and uniformly smooth  Banach space. Then, for all $u_1, u_2 \in U$, following result holds:

\begin{enumerate}
\item $\Delta_p(u_1, u_2) \geq 0$ and $\Delta_p(u_1, u_2) = 0$ if and only if $u_1 = u_2$. 

\item If $U$ is $p$ convex, then we have  \begin{equation} \label{eq1} \begin{split} \Delta_p(u_1, u_2) \geq \frac{C_p}{p}\|u_1-u_2\|^p,  \end{split}
\end{equation} where   $C_p > 0$ is some constant. 

\item If $U^*$ is $q$ smooth,  then we have 
 \begin{equation} \label{eq1} \begin{split} \Delta_q(u_1^*, u_2^*) \leq \frac{G_q}{q}\|u_1^*-u_2^*\|^q, \quad  \forall \  u_1^*, u_2^* \in U^*,  \end{split}
\end{equation} where  $G_q > 0$ is some constant.
\item Following are equivalent:
\begin{enumerate}
\item $\lim_{n \to \infty} \|u_n-u\| = 0$, \item  $\lim_{n \to \infty}  \Delta_p(u_n, u) = 0$ and \item $\lim_{n \to \infty} \|u_n\| = \|u\|$ and $\lim_{n \to \infty} \langle J_p(u_n), u\rangle = \langle J_p(u),u\rangle$. 
\end{enumerate}
\end{enumerate}

\end{theorem}
Proof of parts $(2)$ and $(3)$ of  Theorem $2.2$ are  discussed after the  Remark $2.4$. In $[24, \text{Theorem}\ 2.5]$, results of the type $(2)$ and $(3)$  are discussed with the following Bregman distance $$\Delta_p'(u_1, u_2) = \frac{1}{p}\|u_2\|^p-\frac{1}{p}\|u_1\|^p-\langle J_p(u_1), u_2-u_1\rangle, \quad u_1, u_2 \in U.$$ Note that the definitions of Bregman distance employed in $[24]$ and in this paper are different, because of the interchange of arguments.

\begin{remark} $[38, \text{Theorem} \ 1]$  Let $\delta_X(\epsilon)$ represents the convexity modulus of a uniformly convex real Banach space $X$. Then, there exists a function $\phi_p \in \mathbb{A}$ such that \begin{equation}
  \|x_1+x_2\|^p \geq \|x_1\|^p+p\langle J_p(x_1), x_2\rangle +\sigma_p(x_1, x_2), \ \ x_1, x_2 \in X,\end{equation} where \begin{equation}
     \sigma_p(x_1, x_2) = p\int_0^1 \frac{\big(\|x_1+tx_2\|\vee \|x_1\|)^p}{t}\phi_p\bigg(\frac{t\|x_2\|}{\|x_1+tx_2\|\vee \|x_1\|}\bigg)\, dt,\end{equation} $($see Remark $2.5)$ and $$\mathbb{A} = \big\{\phi:\mathbb{R}^+\to \mathbb{R}^+: \phi(0) = 0, \ \phi(t)\ \text{is strictly increasing and}\ K\ \text{is a positive} $$ $$ \text{ constant such that} \ \phi(t)\geq K\delta_X(t/2)\big\}.$$ Here  $x \wedge y = \min(x, y)$  and $x \vee y = \max(x, y)$ for arbitrarily real numbers $x$ and $y$.  
Since, $\phi_p \in \mathbb{A}$, $(2.6)$ can be written as \begin{equation*}
     \sigma_p(x_1, x_2) \geq  pK_p\int_0^1 \frac{\big(\|x_1+tx_2\|\vee \|x_1\|)^p}{t}\delta_X\bigg(\frac{t\|x_2\|}{2(\|x_1+tx_2\|\vee \|x_1\|)}\bigg)\, dt,\end{equation*}with \begin{equation}
          K_p = 4(2+\sqrt{3})\min\bigg\{\frac{1}{2}p(p-1)\wedge 1, \bigg(\frac{1}{2}p\wedge 1\bigg)(p-1), \hspace{30mm}$$ $$ (p-1)[1-(\sqrt{3}-1)^q), 1-\bigg[1+\frac{(2-\sqrt{3})p}{p-1}\bigg]^{1-p}\bigg\}, \end{equation}  where the value of $K_p$ is obtained from Lemma $3$ in $[38]$. Also if $X$ is $p$ convex, then last inequality can  be written as \begin{equation*}
     \sigma_p(x_1, x_2) \geq  pYK_p\int_0^1 \frac{\big(\|x_1+tx_2\|\vee \|x_1\|)^p}{t}\bigg(\frac{t\|x_2\|}{2(\|x_1+tx_2\|\vee \|x_1\|)}\bigg)^p\, dt\end{equation*} \begin{equation*}
     = p\bigg(\frac{ YK_p}{2^p}\bigg)\|x_2\|^p\int_0^1 t^{p-1}\, dt = C_p\|x_2\|^p,\end{equation*} for some positive constants $Y$ and $C_p =\frac{ YK_p}{2^p}$. Above inequality and $(2.5)$ imply that  \begin{equation}
 \frac{1}{p} \|x_1+x_2\|^p- \frac{1}{p}  \|x_1\|^p  -\langle J_p(x_1), x_2\rangle \geq \frac1p\sigma_p(x_1, x_2)\geq \frac{C_p}{p}\|x_2\|^p.\end{equation}
\end{remark}
Now we come to the proof of part $(2)$ of Theorem $2.2$. In our notations, if we consider $X = U$, $x_1 = u_2$ and $x_2 = u_1-u_2$,  then $(2.8)$ implies that \begin{equation*} \Delta_p(u_1, u_2) = 
 \frac{1}{p} \|u_1\|^p- \frac{1}{p}  \|u_2\|^p  -\langle J_p(u_2), u_1-u_2\rangle \geq \frac{C_p}{p}\|u_1-u_2\|^p,\end{equation*} which is the desired inequality. Here $C_p$ is a constant depending on $p$. Part $(3)$ can be proved similarly by using Theorem $2$ in $[38]$.
 \begin{remark}
 In $[38,  \text{equation}\ 2.2]$, value of $\sigma_p$ given in the statement is \begin{equation*}
     \sigma_p(x, y) = p\int_0^1 \frac{\big(\|x+ty\|\vee \|x\|)^p}{t}\phi_p\bigg(\frac{t\|y\|}{(\|x+ty\|\vee \|y\|)}\bigg)\, dt.\end{equation*} But the actual value is \begin{equation*}
     \sigma_p(x, y) = p\int_0^1 \frac{\big(\|x+ty\|\vee \|x\|)^p}{t}\phi_p\bigg(\frac{t\|y\|}{(\|x+ty\|\vee \|x\|)}\bigg)\, dt,\end{equation*} which can be easily verified from the proof given there.
 \end{remark}

\section{Convergence and convergence rates} 

In the present section, we analyze the convergence and its rates for the iteratively regularized Landweber iteration method $(2.1)$-$(2.2)$. Here, we consider the notation $$B = B_{\rho}^{\Delta}(u^{\dagger}) := \{ u \in U: \Delta_p^{u_0}(u^{\dagger}, u ) \leq \rho^2\},$$ where $\Delta_p^{u_0}(u^{\dagger}, u ) = \Delta_p(u^{\dagger}-u_0, u-u_0)$, $\rho > 0$ is some constant and $u^{\dagger}$ is the solution of  $F(u)=v$ which may not  be unique. We assume $B \subset D(F)$. To prove the main results of the paper, we need to have certain assumptions accumulated below. 
\begin{assum}{} \leavevmode

\begin{enumerate}
\item $U$ is  $q$ smooth and $p$ convex  with $\frac{1}{p}+\frac{1}{q} = 1$, where $p, q > 1$.  

\item $F$ has a Fr\'echet derivative  $F'(\cdot)$ and it satisfies the following local  estimate \begin{equation} \label{eq1} \begin{split} \|F'(u_1)-F'(u_2)\| \leq L \|u_1-u_2\|, \quad \forall\  u_1, u_2 \in B, \end{split}\end{equation} where $L>0$ is a constant. 
\item $F'(\cdot)$ satisfies the boundedness  condition, i.e.  $\|F'(u)\| \leq \hat{L}$ for all $u \in B$ for some positive constant $\hat{L}$. 
\item $F$ is weakly sequentially closed. 

\item Elements in $B$ satisfy the following H\"older stability estimate \begin{equation} \label{eq1} \begin{split} \Delta_p^{u_0}(u_1, u_2) \leq C_F^p \|F(u_1)-F(u_2) \|^{\frac{1+\epsilon}{2}p}, \quad u_1, u_2 \in B, \ \epsilon \in (0, 1], \end{split}\end{equation} where $C_F > 0$ is a constant. 
\item  $u_0$ lies in $B$ and there exists a sequence $\{r_k\}_{k\in \mathbb{N}\cup \{0\}}$ such that $\Delta_p^{u_0}(u_k^{\delta}, u^{\dagger})\geq r_k( \Delta_p^{u_0}(u^{\dagger}, u_k^{\delta}))$ for each $k$. For example in Hilbert spaces, we have $$\Delta_2^{u_0}( u^{\dagger}, u_k^{\delta})=\frac{1}{2}\|u^{\dagger}- u_k^{\delta}\|^2 = \Delta_2^{u_0}(u_k^{\delta}, u^{\dagger}),$$ which means $r_k=1$ for each $k$.
\item The sequence $\{\beta_k\}$ satisfies $(2.2)$, $\sum_{k}\beta_k < \infty$ and $\beta_{\max}$ is sufficiently small.
\item  $\mu$ is chosen such that \begin{equation} \label{eq1} \begin{split} \mu^{q-1} < \frac{q}{2^q \hat{L}^q G_q}. \qquad  \end{split} 
\end{equation}
\item  $\rho^2$ satisfies \begin{equation} \label{eq1} \begin{split} \rho^2 = \hat{L}^{-p}(LC_F^2)^{\frac{-p}{\epsilon}}  \bigg( \frac{C_p}{p}\bigg)^{1+\frac{2}{\epsilon}}. \end{split} 
\end{equation}
\item a-priori choice of the stopping index $k_{\star}$ is $$k_{\star}(\delta)=\min\{k \in \mathbb{N}:\ \beta_k\leq \tau\delta\},$$with $\tau >0$ sufficiently large.
\end{enumerate}
\end{assum}
\begin{remark}
The H\"older type stability estimate $(3.2)$ for the  special case $p=2$ can be obtained by a lower bound on the Fr\'echet derivative $F'$. Let there exists a constant $K > 0$ such that $$\bigg\|F'(u)\bigg(\frac{u-u^{\dagger}}{\|u-u^{\dagger}\|}\bigg)\bigg\| \geq K \|u-u^{\dagger}\|^{1-\epsilon_1} \quad \forall u \in D(F) \cap B_r(u^{\dagger}),$$  where $B_r(u^{\dagger})$ is some ball of radius $r$ $($sufficiently small$)$ around $u^{\dagger}$ and $\epsilon_1 \in (0, 1]$. The last inequality and the estimate $$\|F(u')-F(u)-F'(u)(u'-u)\| \leq \frac{L}{2}\|u'-u\|^2 \quad \forall\ u, u' \in D(F),$$imply that $$K\|u-u^{\dagger}\|^{2-\epsilon_1} \leq\|F(u)-F(u^{\dagger})-F'(u)(u-u^{\dagger})\|+\|F(u)-F(u^{\dagger})\|$$ $$ \leq \frac{L}{2}\|u-u^{\dagger}\|^2 + \|F(u)-F(u^{\dagger})\|, \quad\forall u \in D(F) \cap B_r(u^{\dagger}).$$ Since   $r$ is  small, last inequality can also be written as $$K\|u-u^{\dagger}\|^{2-\epsilon_1} \leq \frac{L}{2}\|u-u^{\dagger}\|^{2-\epsilon_1} + \|F(u)-F(u^{\dagger})\| $$which immediately leads to the estimate $$\|u-u^{\dagger}\| \leq C'\|F(u)-F(u^{\dagger})\|^{\frac{1}{2-\epsilon_1}} \quad \forall u \in D(F) \cap B_r(u^{\dagger}),$$ where $C'$ is a constant depending on $K$ and $L$. Since,  in the case of Hilbert spaces,  $p=2$ and $\Delta_2^{u_0}(u, u^{\dagger}) = \frac{1}{2}\|u-u^{\dagger}\|^2$ where $u \in U$, an estimate of the type $(3.2)$ can be obtained. In general, it is impossible to obtain a lower bound for $F'$ due to ill-posedness of almost all the inverse problems. This lower bound has been studied for many inverse problems under various assumptions, see, for instance,  $[7, 13]$. The key fact used in $[7, 13]$  to obtain the lower bounds  is  that the forward operator has been projected  properly.
\end{remark}
\begin{remark}
  For $\epsilon = 1$ and  $u_0 = 0$ in $(3.2)$ $($observe that $(3.2)$ with  $\epsilon = 1$ is  the Lipschitz-type stability estimate$)$,  we have $$\langle J_p(u^{\dagger}), u-u^{\dagger}\rangle \leq \|u^{\dagger}\|^{p-1}\|u-u^{\dagger}\| \leq \bigg(\frac{p}{C_p}\bigg)^{\frac{1}{p}} \|u^{\dagger}\|^{p-1}\Delta_p(u, u^{\dagger})^{\frac{1}{p}} $$ $$\leq K \|F(u)-F(u^{\dagger})\| \quad \forall u \in B,$$ where above holds by using $(2.3)$ and  $K = C_F\|u^{\dagger}\|^{p-1}\big(\frac{p}{C_p}\big)^{\frac{1}{p}}$. In $[26]$, it is shown that   the last inequality implies the source condition $J_p(u^{\dagger}) = F'(u^{\dagger})^*v$ for some $v$ such that $\|v\| \leq 1$.
\end{remark}
Now, we are ready to state our main result in which we obtain the convergence and  its rates with some additional assumptions on the sequence $\{\beta_k\}$. 
\begin{theorem}
Let $F$ be a non-linear operator between the Banach spaces $U$, $V$ and the operator equation $F(u) = v, v \in V$, has  a solution $u^{\dagger}$. Suppose that the Assumption $3.1$ holds,  $v^{\delta}\in V$ be such that $\|v^{\delta}-v\|\leq \delta$. Then all the iterates $u_{k+1}^{\delta}$  of iteratively regularized Landweber iteration  method $(2.1)$-$(2.2)$ remain in $B$ for all $k\leq k_{\star}(\delta)-1$ provided $\beta_{\max}$ is sufficiently small $($see Lemma $3.2$ for  exact estimate of $\beta_{\max})$. Moreover, iterates satisfy the recurrence relation
$$\gamma_{k+1}^{\delta} \leq \gamma_k^{\delta}+ K_1\delta^p+K_2\delta^{\epsilon}+K_3\delta^{p+\epsilon}+K_4\delta-K_6\rho^2,$$  for some constants $K_i, 1\leq i\leq 4$, and $K_6>0$, where $\gamma_k = \Delta_p^{u_0}(u^{\dagger}, u_k^{\delta})$. We also obtain the convergence rates $$\Delta_p^{u_0}(u^{\dagger}, u_{k_{\star}})-(1-K_6)\rho^2=O(\delta^{\epsilon}), \ \text{as}\ \delta \to 0.$$ For $\delta=0$, iterates $u_{k+1}$   of iteratively regularized Landweber iteration  method $(2.1)$-$(2.2)$ not only remain in $B$ but also converge to the solution $u^{\dagger}$. Further,  we get the following  rates:  \begin{enumerate}
\item  Iterates $\gamma_k = \Delta_p^{u_0}(u^{\dagger}, u_k)$ satisfy the recursion formula \begin{equation}\label{eq1} \begin{split} \gamma_{k+1} \leq  - K_8\gamma_k^{\frac{2}{1+\epsilon}} + \alpha_k \gamma_k + K_{11} \beta_k,      \end{split}  
\end{equation} for some positive constants $K_8, K_{11}$  and $\{\alpha_k\}$ is a sequence converges to $1$.  Further, if $\{\beta_k\}$ satisfies $\beta_k \leq C\gamma_k$ $($smoothness condition$)$ for some constant $C > 0$, then the convergence rate, for $\epsilon \in (0,1)$,  is given by  \begin{equation*}\label{eq1} \begin{split}\Delta_p^{u_0}(u^{\dagger}, u_k) \leq   \bigg(\big( g_k \rho^2\big)^{-\frac{1-\epsilon}{1+\epsilon}}  + h_k \bigg)^{-\frac{1+\epsilon}{1-\epsilon}}, \quad k =  1, 2, \ldots    \end{split}  
\end{equation*} where \begin{equation*}\hspace{10mm}
g_k = \prod_{i=0}^{k-1}d_i,  \  k \geq 1,\ \text{and}\ h_k = \sum_{j=1}^{k-1}\bigg(d_jd_{j+1}\ldots d_{k-1}\bigg)^{-\frac{1+\epsilon}{1-\epsilon}}f_{j-1} + f_{k-1}, \ k \geq 2,  \ h_1 = f_0, 
\end{equation*}  with $f_k= te_k d_k ^{-t}$, $d_k = \alpha_k +CK_{11}$, $e_k = \frac{K_8}{d_k}$ and  $t=\frac{1-\epsilon}{1+\epsilon}$. \newline
For $\epsilon = 1$, we get \begin{equation*}\label{eq1} \begin{split} \Delta_p^{u_0}(u^{\dagger}, u_k)   \leq \prod_{i=0}^{k-1}(- K_8 + \alpha_i + K_{11}C) \rho^2, \quad k =   1, 2,  \ldots    \end{split}  
\end{equation*}
\item  we  also obtain the rate \begin{equation*}\label{eq1} \begin{split} \Delta_p^{u_0}(u^{\dagger}, u_k) =  O(\beta_k^{q-1}), \quad \text{as} \  k \to \infty, \end{split}  
\end{equation*} provided  \begin{equation*}\label{eq1} \begin{split}    K_{12} +  \eta \beta_k^{-1}\bigg[ \alpha_k      -  \bigg( \frac{\beta_{k+1}}{\beta_{k}}\bigg)^{q-1}\bigg] \leq 0,  \end{split}  
\end{equation*} for some constants $\eta , K_{12}$,  and  a sequence $\{\alpha_k\}$ converging to $1$. 
\end{enumerate}
\end{theorem}
Instead of giving a single proof of Theorem $3.1$, we discuss it in parts in the form of a series of lemmas to have a better understanding. 
 In the first lemma, we obtain an estimate of $\Delta_p^{u_0}(u^{\dagger}, u_{k+1}^{\delta}) - \Delta_p^{u_0}( u^{\dagger}, u_{k}^{\delta})$.
\begin{lemma}
Let $F$ be a non-linear operator between the Banach spaces $U$, $V$ and the operator equation $F(u) = v, v \in V$, has  a solution $u^{\dagger}$. Suppose that  Assumption $3.1$ holds and $v^{\delta}\in V$ be such that $\|v^{\delta}-v\|\leq \delta$. Then the iterates $u_{k+1}^{\delta}$  of iteratively regularized Landweber iteration  method $(2.1)$-$(2.2)$ satisfy the following inequality \begin{equation*}   \Delta_p^{u_0}(u^{\dagger}, u_{k+1}^{\delta}) - \Delta_p^{u_0}( u^{\dagger}, u_{k}^{\delta}) \leq  \bigg( 2^{q-1} \frac{G_q}{q}  \mu^q \hat{L}^q - \mu\bigg)\|  F(u_k^{\delta})- v^{\delta}\|^{p}  + \frac{\mu }{2} LC_F^2\bigg(\frac{p}{C_p}\bigg)^{2/p}\|F(u_k^{\delta})-v^{\delta}\|^{p+\epsilon}   \end{equation*}    \begin{equation*} + \bigg( 2^{p+q-2}  \frac{G_q}{q} \beta_k^q + \beta_k  \frac{(p-1)\epsilon_2^{\frac{p}{p-1}}}{p}\bigg) \| u^{\dagger}-u_0 \|^p    +\mu  \|F(u_k^{\delta})-v^{\delta}\|^{p-1}\delta+  
 \bigg(\frac{\beta_k \epsilon_2^{-p}}{ C_p} -  \end{equation*} \begin{equation*}    (1+r_k) \beta_k + 2^{p+q-2}\beta_k^q \frac{G_q}{q} \frac{p}{C_p}   \bigg) \gamma_k^{\delta},\quad \text{where}\ \epsilon_2>0.\end{equation*} 
\end{lemma}
\begin{proof}
From  Lemma $2.1$ and  $(2.1)$, we can write \begin{equation*} \label{eq1}  \begin{split} \Delta_p^{u_0}(u^{\dagger}, u_{k+1}^{\delta}) - \Delta_p^{u_0}( u^{\dagger}, u_{k}^{\delta})   = \Delta_p^{u_0}(u_{k}^{\delta}, u_{k+1}^{\delta}) + \langle J_p(u_k^{\delta}-u_0) - J_p(u_{k+1}^{\delta}-u_0), u^{\dagger}-u_k^{\delta} \rangle\hspace{5mm}\end{split} 
\end{equation*}
\begin{equation}  \label{eq1}  \begin{split}  = \Delta_p^{u_0}(u_{k}^{\delta}, u_{k+1}^{\delta}) - \mu \langle j_p(F(u_k^{\delta})-v^{\delta}), F'(u_k^{\delta})(u_k^{\delta}-u^{\dagger})\rangle   + \beta_k\langle J_p(u^{\dagger}-u_0), u^{\dagger}-u_k^{\delta}\rangle \\- \beta_k \langle J_p(u^{\dagger}-u_0) - J_p(u_k^{\delta}-u_0), u^{\dagger}-u_k^{\delta} \rangle. \end{split} 
\end{equation} Now,  we estimate  each of the four terms of the right side to $(3.6)$  individually. For the first term, using Definitions $2.5$  and $2.1$,   we have \begin{equation*} \label{eq1}  \begin{split}  \Delta_p^{u_0}(u_{k}^{\delta}, u_{k+1}^{\delta}) =\frac{1}{p}\|u_{k}^{\delta}-u_0\|^p- \frac{1}{p}\|u_{k+1}^{\delta}-u_0\|^p-\langle J_p(u_{k+1}^{\delta}-u_0), u_{k}^{\delta}- u_{k+1}^{\delta}\rangle\hspace{30mm}\\   =\frac{1}{q}\|u_{k+1}^{\delta}-u_0\|^p- \frac{1}{q}\|u_{k}^{\delta}-u_0\|^p-\langle J_p(u_{k+1}^{\delta}-u_0), u_{k}^{\delta}- u_{k+1}^{\delta}\rangle+\|u_{k}^{\delta}-u_0\|^p-\|u_{k+1}^{\delta}-u_0\|^p \\   =\frac{1}{q}\|u_{k+1}^{\delta}-u_0\|^p- \frac{1}{q}\|u_{k}^{\delta}-u_0\|^p-\langle J_p(u_{k+1}^{\delta}-u_0), (u_{k}^{\delta}-u_0)- (u_{k+1}^{\delta}-u_0)\rangle\hspace{26mm}\\ + \langle u_{k}^{\delta}-u_0, J_p(u_{k}^{\delta}- u_{0})\rangle  -  \langle u_{k+1}^{\delta}-u_0, J_p(u_{k+1}^{\delta}- u_{0})\rangle \\ = \frac{1}{q}\|u_{k+1}^{\delta}-u_0\|^p- \frac{1}{q}\|u_{k}^{\delta}-u_0\|^p + \langle u_{k}^{\delta}-u_0, J_p(u_{k}^{\delta}- u_{0})  -   J_p(u_{k+1}^{\delta}- u_{0})\rangle\hspace{25mm}\\ = \frac{1}{q}\|J_p(u_{k+1}^{\delta}-u_0)\|^q- \frac{1}{q}\|J_p(u_{k}^{\delta}-u_0)\|^q - \langle u_{k}^{\delta}-u_0, J_p(u_{k+1}^{\delta}- u_{0})  -   J_p(u_{k}^{\delta}- u_{0})\rangle\hspace{10mm}\\ = \Delta_q(J_p(u_{k+1}^{\delta}-u_0), J_p(u_{k}^{\delta}-u_0)).\hspace{30mm} \end{split}
\end{equation*}
Use $(2.4)$ and  then $(2.1)$ in above to obtain
\begin{equation}\begin{split} \Delta_p^{u_0}(u_{k}^{\delta}, u_{k+1}^{\delta})
\leq 
  \frac{G_q}{q} \|J_p(u_{k+1}^{\delta}-u_0)-J_p(u_k^{\delta}-u_0)\|^{q} \hspace{22mm}\\ =  \frac{G_q}{q} \|\beta_k J_p(u_{k}^{\delta}-u_0)+ \mu F'(u_k^{\delta})^* j_p(F(u_k^{\delta})- v^{\delta})\|^{q}.  \end{split}
\end{equation}  
Now using the  estimate  \begin{equation*} \label{eq1} \begin{split} \|u_1+u_2\|^r \leq  2^{r-1}(\|u_1\|^r+\|u_2\|^r),\ \ r\geq 1,\ \ u_1, u_2 \in U, \end{split} 
\end{equation*}see $[26, \text{Lemma} \ 3.20]$, twice into $(3.7)$,   we have  
\begin{equation*} \label{eq1}  \begin{split}  \Delta_p^{u_0}(u_{k}^{\delta}, u_{k+1}^{\delta}) \leq  2^{q-1} \frac{G_q}{q} \bigg( \beta_k^q  \| J_p(u_{k}^{\delta}-u_0) \|^q+ \mu^q \| F'(u_k^{\delta})^* j_p(F(u_k^{\delta})- v^{\delta})\|^{q} \bigg) \hspace{20mm}\\ =  2^{q-1} \frac{G_q}{q} \bigg( \beta_k^q  \| u_{k}^{\delta}-u_0 \|^p+ \mu^q \| F'(u_k^{\delta})^* j_p(F(u_k^{\delta})- v^{\delta})\|^{q} \bigg)\hspace{25mm} \\ \leq  2^{q-1} \frac{G_q}{q} \bigg( 2^{p-1}  \beta_k^q \big( \| u^{\dagger}-u_0 \|^p + \|  u^{\dagger} - u_k^{\delta}\|^p\big)+ \mu^q \| F'(u_k^{\delta})^* j_p(F(u_k^{\delta})- v^{\delta})\|^{q} \bigg) \end{split} 
\end{equation*} \begin{equation} \label{eq1} \begin{split} \hspace{20mm}  \leq  2^{q-1} \frac{G_q}{q} \bigg( 2^{p-1}  \beta_k^q \big( \| u^{\dagger}-u_0 \|^p + \frac{p}{C_p} \Delta_p^{u_0}(u^{\dagger}, u_{k}^{\delta}) \big)+ \mu^q \hat{L}^q\|  F(u_k^{\delta})- v^{\delta}\|^{p} \bigg), \end{split} 
\end{equation} where the last inequality is obtained by incorporating $(2.3)$ and  $(3)$ of Assumption $3.1$  provided $u_k^{\delta}$ satisfies the estimate $(3.2)$ which  will be shown later.  
\newline 
Next, let us estimate  the  second term on the right side to $(3.6)$ as $$
 - \mu \langle j_p(F(u_k^{\delta})-v^{\delta}), F'(u_k^{\delta})(u_k^{\delta}-u^{\dagger})\rangle \qquad \qquad \qquad \qquad \qquad \qquad \qquad \qquad \qquad \qquad \qquad $$
  \begin{equation*}\label{eq1}    = -\mu \langle j_p(F(u_k^{\delta})-v^{\delta}), F(u_k^{\delta})-v^{\delta} \rangle   + \mu \langle j_p(F(u_k^{\delta})-v^{\delta}), F(u_k^{\delta})-v^{\delta} - F'(u_k^{\delta})(u_k^{\delta}-u^{\dagger})\rangle   
\end{equation*}  \begin{equation*}\label{eq1}    = -\mu \|F(u_k^{\delta})-v^{\delta}\|^p   + \mu \langle j_p(F(u_k^{\delta})-v^{\delta}), F(u_k^{\delta})-v^{\delta} - F'(u_k^{\delta})(u_k^{\delta}-u^{\dagger})\rangle. 
\end{equation*} By employing fundamental theorem of calculus for $F'(\cdot)$, i.e. $$\|F(u_k^{\delta})-v^{\delta}-F'(u_k^{\delta})(u_k^{\delta}-u^{\dagger})\|\leq \frac{L}{2}\|u_k^{\delta}-u^{\dagger}\|^2+\delta,$$ and  $(3.1)$ in the last equality to obtain $$
 - \mu \langle j_p(F(u_k^{\delta})-v^{\delta}), F'(u_k^{\delta})(u_k^{\delta}-u^{\dagger})\rangle \qquad \qquad \qquad \qquad \qquad \qquad \qquad \qquad \qquad \qquad \qquad $$ \begin{equation*}\label{eq1}    \leq -\mu \|F(u_k^{\delta})-v^{\delta}\|^p   + \frac{\mu L}{2}  \|F(u_k^{\delta})-v^{\delta}\|^{p-1}\|u_k^{\delta}-u^{\dagger}\|^2+ \mu  \|F(u_k^{\delta})-v^{\delta}\|^{p-1}\delta.
\end{equation*} Using $(2.3)$ and then  $(3.2)$, we further estimate $$
 - \mu \langle j_p(F(u_k^{\delta})-v^{\delta}), F'(u_k^{\delta})(u_k^{\delta}-u^{\dagger})\rangle \qquad \qquad \qquad \qquad \qquad \qquad \qquad \qquad \qquad \qquad \qquad $$ \begin{equation}\label{eq1}    \leq -\mu \|F(u_k^{\delta})-v^{\delta}\|^p   + \frac{\mu }{2} LC_F^2\bigg(\frac{p}{C_p}\bigg)^{2/p} \|F(u_k^{\delta})-v^{\delta}\|^{p+\epsilon}+\mu  \|F(u_k^{\delta})-v^{\delta}\|^{p-1}\delta.\end{equation}
 Now, let us turn to estimate the third term of the right side to $(3.6)$ as 
\begin{equation*} \label{eq1}  \begin{split}   \beta_k\langle J_p(u^{\dagger}-u_0), u^{\dagger}-u_k^{\delta}\rangle \leq \beta_k|\langle J_p(u^{\dagger}-u_0), u^{\dagger}-u_k^{\delta}\rangle| \\ \leq \beta_k\| J_p(u^{\dagger}-u_0)\| \| u^{\dagger}-u_k^{\delta}\| \\ = \beta_k \|u^{\dagger}-u_0\|^{p-1} \| u^{\dagger}-u_k^{\delta}\|. \end{split} 
\end{equation*} Thanks to Young's inequality $ab \leq \frac{a^r}{r} + \frac{b^s}{s}$ with H\"older conjugates $r, s$ for  $a = \epsilon_2\|u^{\dagger}-u_0\|^{p-1}$, $b = \epsilon_2^{-1}\| u^{\dagger}-u_k^{\delta}\|$, $r = \frac{p}{p-1}$, $s = p$,  $\epsilon_2>0$, and $(2.3)$ to further  yield \begin{equation} \label{eq1}  \begin{split}   \beta_k\langle J_p(u^{\dagger}-u_0), u^{\dagger}-u_k^{\delta}\rangle \leq \beta_k \bigg( \frac{(p-1)\epsilon_2^{\frac{p}{p-1}}}{p}\|u^{\dagger}-u_0\|^{p} + \frac{\epsilon_2^{-p}}{p}\| u^{\dagger}-u_k^{\delta}\|^p \bigg) \\ \leq \beta_k \bigg( \frac{(p-1)\epsilon_2^{\frac{p}{p-1}}}{p}\|u^{\dagger}-u_0\|^{p} + \frac{\epsilon_2^{-p}}{ C_p} \Delta_p^{u_0}( u^{\dagger}, u_k^{\delta}) \bigg).  \end{split} 
\end{equation}
Finally, using Lemma $2.1$ in the fourth term on the right hand side to $(3.6)$ to obtain 
\begin{equation*}\label{eq1}  
- \beta_k \langle J_p(u^{\dagger}-u_0) - J_p(u_k^{\delta}-u_0), u^{\dagger}-u_k^{\delta} \rangle = - \beta_k \Delta_p^{u_0}(u^{\dagger}, u_k^{\delta}) - \beta_k \Delta_p^{u_0}( u_k^{\delta}, u^{\dagger})+ \beta_k \Delta_p^{u_0}( u_k^{\delta}, u_k^{\delta})
\end{equation*} \begin{equation}\label{eq1}  
\hspace{20mm}\leq - \beta_k \Delta_p^{u_0}(u^{\dagger}, u_k^{\delta}) - \beta_k \Delta_p^{u_0}( u_k^{\delta}, u^{\dagger})\leq - \beta_k(1+r_k) \Delta_p^{u_0}(u^{\dagger}, u_k^{\delta}),
\end{equation} where the last inequality holds because of  $(6)$ of Assumption $3.1$. 
Inserting all the estimates $(3.8)$-$(3.11)$ into $(3.6)$ and use the notation $\gamma_k^{\delta}  = \Delta_p^{u_0}(u^{\dagger}, u_k^{\delta})$, we have
\begin{equation*} \gamma_{k+1}^{\delta} - \gamma_k^{\delta}
   \leq  2^{q-1} \frac{G_q}{q} \bigg( 2^{p-1}  \beta_k^q \big( \| u^{\dagger}-u_0 \|^p + \frac{p}{C_p} \gamma_k^{\delta} \big)+ \mu^q \hat{L}^q\|  F(u_k^{\delta})- v^{\delta}\|^{p} \bigg)   -\mu \|F(u_k^{\delta})-v^{\delta}\|^p \end{equation*}\begin{equation*}
       + \frac{\mu }{2} LC_F^2\bigg(\frac{p}{C_p}\bigg)^{2/p} \|F(u_k^{\delta})-v^{\delta}\|^{p+\epsilon}  +  \beta_k \bigg( \frac{(p-1)\epsilon_2^{\frac{p}{p-1}}}{p}\|u^{\dagger}-u_0\|^{p}\bigg)   +  \frac{\beta_k\epsilon_2^{-p}}{ C_p} \gamma_k^{\delta}  \end{equation*}\begin{equation*}  - \beta_k(1+r_k) \gamma_k^{\delta}+ \mu  \|F(u_k^{\delta})-v^{\delta}\|^{p-1}\delta 
\end{equation*}
\begin{equation*}  = \bigg( 2^{q-1} \frac{G_q}{q}  \mu^q \hat{L}^q - \mu\bigg)\|  F(u_k^{\delta})- v^{\delta}\|^{p} + \bigg( 2^{p+q-2}  \frac{G_q}{q} \beta_k^q + \beta_k  \frac{(p-1)\epsilon_2^{\frac{p}{p-1}}}{p}\bigg) \| u^{\dagger}-u_0 \|^p     + \frac{\mu }{2} LC_F^2\bigg(\frac{p}{C_p}\bigg)^{2/p}  \end{equation*} \begin{equation}   \times\|F(u_k^{\delta})-v^{\delta}\|^{p+\epsilon}  + \bigg(\frac{\beta_k \epsilon_2^{-p}}{ C_p}     -(1+r_k) \beta_k + 2^{p+q-2}\beta_k^q \frac{G_q}{q} \frac{p}{C_p}   \bigg) \gamma_k^{\delta}+\mu  \|F(u_k^{\delta})-v^{\delta}\|^{p-1}\delta.   
\end{equation}
\end{proof}
\begin{remark}
We have intentionally introduced the parameter $\epsilon_2$ in $(3.10)$. Rationale behind the introduction of this $\epsilon_2$ is discussed in Remark $3.5$.
\end{remark}
In the next lemma, we show that all the iterates of our iteration scheme remain in $B$ using  Lemma $3.1$ under certain assumptions. 
\begin{lemma}
Suppose that all the assumptions of Lemma $3.1$ hold. Then, all the iterates of $(2.1)$-$(2.2)$ remain in $B$ for all $k\leq k_*(\delta)-1$   provided $\tau$ defined in $(10)$ of Assumption $3.1$ is sufficiently large and $\beta_{\max}$ is such that $$\beta_{\max} <  \sqrt[q-1]{\frac{1}{2^{p+q-1}} \frac{C_p}{p}\frac{q}{G_q}\bigg(1+r_k-\frac{(p-1)\epsilon_2^{\frac{p}{p-1}}}{C_p}-\frac{\epsilon_2^{-p}}{C_p}\bigg)}.$$
\end{lemma}
\begin{proof}
Let us assume that $u_k^{\delta} \in B$, and then applying $(3.3)$ in the first term on the right side to $(3.12)$, we get \begin{equation*} \gamma_{k+1}^{\delta} - \gamma_k^{\delta}  \leq -\frac{\mu}{2}\|  F(u_k^{\delta})- v^{\delta}\|^{p} + \bigg( 2^{p+q-2}  \frac{G_q}{q} \beta_k^q + \beta_k  \frac{(p-1)\epsilon_2^{\frac{p}{p-1}}}{p}\bigg) \| u^{\dagger}-u_0 \|^p+\mu  \|F(u_k^{\delta})-v^{\delta}\|^{p-1}\delta  \end{equation*}  \begin{equation}  + \frac{\mu }{2} LC_F^2\bigg(\frac{p}{C_p}\bigg)^{2/p} \|F(u_k^{\delta})-v^{\delta}\|^{p+\epsilon}  + \bigg(\frac{\beta_k\epsilon_2^{-p}}{ C_p}     - \beta_k(1+r_k)+ 2^{p+q-2}\beta_k^q \frac{G_q}{q} \frac{p}{C_p}   \bigg) \gamma_k^{\delta}\hspace{0mm}.   
\end{equation} From the mean value inequality,  $(3)$ of Assumption $3.1$, $(2.3)$ and $(3.4)$,   we get \begin{equation*}  \|F(u_k^{\delta})-v^{\delta}\| = \|F(u_k^{\delta})-F(u^{\dagger})\|+\delta \leq \hat{L}\|u_k^{\delta}-u^{\dagger}\| +{\delta} \leq \hat{L} \bigg(\frac{p}{C_p}\bigg)^{1/p} \Delta_p^{u_0}(u^{\dagger}, u_k^{\delta})^{1/p}+{\delta}  \end{equation*} \begin{equation} \leq \hat{L} \bigg(\frac{p}{C_p}\bigg)^{1/p}\rho^{\frac{2}{p}}+{\delta} \leq \bigg(\frac{C_p}{p}\bigg)^{\frac{2}{p\epsilon}}(LC_F^2)^{\frac{-1}{\epsilon}}+{\delta}. 
\end{equation}Above with the inequality $$(r+s)^{\epsilon}\leq r^{\epsilon}+s^{\epsilon} \ \text{for}\ r, s \geq 0,\ {\epsilon}\in [0, 1],$$ implies that
\begin{equation*} \|F(u_k^{\delta})-v^{\delta}\|^{\epsilon}  \leq \bigg[\bigg(\frac{C_p}{p}\bigg)^{\frac{2}{p\epsilon}}(LC_F^2)^{\frac{-1}{\epsilon}}+{\delta}\bigg]^{\epsilon} \leq \bigg(\frac{C_p}{p}\bigg)^{\frac{2}{p}}(LC_F^2)^{-1}+\delta^{\epsilon}.   
\end{equation*}
This inequality further leads to the estimate
\begin{equation*}\label{eq1} \begin{split}
 \hspace{-50mm}-\frac{\mu}{2}\|  F(u_k^{\delta})- v^{\delta}\|^{p} +  \frac{\mu }{2} LC_F^2\bigg(\frac{p}{C_p}\bigg)^{2/p} \|F(u_k^{\delta})-v^{\delta}\|^{p+\epsilon} \end{split}  \hspace{20mm}
\end{equation*} \begin{equation}  = 
 \|F(u_k^{\delta})-v^{\delta}\|^{p} \bigg[-\frac{\mu}{2} +  \frac{\mu}{2} LC_F^2\bigg(\frac{p}{C_p}\bigg)^{2/p} \|F(u_k^{\delta})-v^{\delta}\|^{\epsilon}\bigg]  \leq \frac{\mu }{2} LC_F^2\bigg(\frac{p}{C_p}\bigg)^{2/p}\delta^{\epsilon}\|  F(u_k^{\delta})- v^{\delta}\|^{p}.   
\end{equation} Employing the estimate$$(r_1+s_1)^{\lambda}\leq 2^{\lambda-1}(r_1^{\lambda}+s_1^{\lambda}) \ \text{for}\ r_1, s_1 \geq 0,\ \lambda \geq 1,$$ in $(3.15)$ after incorporating   $(3.14)$ in it to obtain \begin{equation*}\label{eq1} \begin{split}
 \hspace{-50mm}-\frac{\mu}{2}\|  F(u_k^{\delta})- v^{\delta}\|^{p} +  \frac{\mu }{2} LC_F^2\bigg(\frac{p}{C_p}\bigg)^{2/p} \|F(u_k^{\delta})-v^{\delta}\|^{p+\epsilon} \end{split}  
\end{equation*} \begin{equation}
\leq 2^{p-1}\frac{\mu }{2} \big(\delta^{\epsilon}\big)\bigg(\frac{C_p}{p}\bigg)^{\frac{2(p-\epsilon)}{p\epsilon}}(LC_F^2)^{\frac{-(p-\epsilon)}{\epsilon}}+ 2^{p-1}\frac{\mu }{2} LC_F^2\bigg(\frac{p}{C_p}\bigg)^{2/p}\delta^{p+\epsilon}.
\end{equation}
Thus, $(3.13)$ and $(3.16)$ imply that \begin{equation}\label{eq1} \begin{split} \gamma_{k+1}^{\delta} - \gamma_k^{\delta}\leq  \bigg( 2^{p+q-2}  \frac{G_q}{q} \beta_k^q + \beta_k  \frac{(p-1)\epsilon_2^{\frac{p}{p-1}}}{p}\bigg) \| u^{\dagger}-u_0 \|^p +2^{p-1}\frac{\mu }{2} \delta^{\epsilon}\bigg(\frac{C_p}{p}\bigg)^{\frac{2(p-\epsilon)}{p\epsilon}}(LC_F^2)^{\frac{-(p-\epsilon)}{\epsilon}} + \hspace{2mm} \\  \bigg(\frac{\beta_k\epsilon_2^{-p}}{ C_p}     - \beta_k(1+r_k) + 2^{p+q-2}\beta_k^q  \frac{G_q}{q} \frac{p}{C_p}   \bigg) \gamma_k^{\delta}+\mu  \|F(u_k^{\delta})-v^{\delta}\|^{p-1}\delta + 2^{p-1}\frac{\mu }{2} LC_F^2\bigg(\frac{p}{C_p}\bigg)^{2/p}\delta^{p+\epsilon}. \end{split}  
\end{equation} Because of the assumption $u_0 \in B$ and $(2.3)$,   estimate $(3.17)$ can  be rewritten as  \begin{equation*}\label{eq1} \begin{split} \gamma_{k+1}^{\delta} - \gamma_k^{\delta}  \leq \bigg[ \frac{p}{C_p} \bigg( 2^{p+q-2}  \frac{G_q}{q} \beta_k^q + \beta_k  \frac{(p-1)\epsilon_2^{\frac{p}{p-1}}}{p}\bigg)     + \bigg(\frac{\beta_k\epsilon_2^{-p}}{ C_p}     - \beta_k(1+r_k) + 2^{p+q-2}\beta_k^q  \frac{G_q}{q} \frac{p}{C_p}   \bigg) \bigg]\rho^2\\ +\mu  \|F(u_k^{\delta})-v^{\delta}\|^{p-1}\delta +2^{p-1}\frac{\mu }{2} \delta^{\epsilon}\bigg(\frac{C_p}{p}\bigg)^{\frac{2(p-\epsilon)}{p\epsilon}}(LC_F^2)^{\frac{-(p-\epsilon)}{\epsilon}}+ 2^{p-1}\frac{\mu }{2} LC_F^2\bigg(\frac{p}{C_p}\bigg)^{2/p}\delta^{p+\epsilon}\hspace{5mm}\end{split} \end{equation*}

\begin{equation}\label{eq1} \begin{split} = \bigg[   2^{p+q-1} \beta_k^q  \frac{G_q}{q}   \frac{p}{C_p}   + \beta_k \bigg(\frac{(p-1)\epsilon_2^{\frac{p}{p-1}}}{C_p}+\frac{\epsilon_2^{-p}}{C_p}-(1+r_k)\bigg)     \bigg]\rho^2+2^{p-1}\frac{\mu }{2} \delta^{\epsilon}\bigg(\frac{C_p}{p}\bigg)^{\frac{2(p-\epsilon)}{p\epsilon}}(LC_F^2)^{\frac{-(p-\epsilon)}{\epsilon}}\\ +\mu  \|F(u_k^{\delta})-v^{\delta}\|^{p-1}\delta+ 2^{p-1}\frac{\mu }{2} LC_F^2\bigg(\frac{p}{C_p}\bigg)^{2/p}\delta^{p+\epsilon}.\hspace{10mm}\end{split} \end{equation}
Now we know that $p> 1$ which means  either $0<p-1\leq 1$ or $p-1>1$. If $p-1\leq 1$, then employ the estimate $(r+s)^{p-1}\leq r^{p-1}+s^{p-1} \ \text{for}\ r, s \geq 0$, otherwise  estimate $(r+s)^{p-1}\leq 2^{p-2}(r^{p-1}+s^{p-1}) \ \text{for}\ r, s \geq 0$ in $(3.14)$ to obtain (we find a single estimate for both the cases) $$\|F(u_k^{\delta})-v^{\delta}\|^{p-1}\leq \max\{2^{p-2}, 1\}\bigg[\bigg(\frac{C_p}{p}\bigg)^{\frac{2(p-1)}{p\epsilon}}(LC_F^2)^{\frac{-(p-1)}{\epsilon}}+\delta^{p-1}\bigg].$$ Employing this estimate  in $(3.18)$ to get \begin{equation*} \gamma_{k+1}^{\delta} - \gamma_k^{\delta}  \leq \bigg[   2^{p+q-1} \beta_k^q  \frac{G_q}{q}   \frac{p}{C_p}   + \beta_k \bigg(\frac{(p-1)\epsilon_2^{\frac{p}{p-1}}}{C_p}+\frac{\epsilon_2^{-p}}{C_p}-(1+r_k)\bigg)     \bigg]\rho^2+2^{p-1}\frac{\mu }{2} \delta^{\epsilon}\bigg(\frac{C_p}{p}\bigg)^{\frac{2(p-\epsilon)}{p\epsilon}}\times \end{equation*}\begin{equation}(LC_F^2)^{\frac{-(p-\epsilon)}{\epsilon}} + K_1\delta\bigg(\frac{C_p}{p}\bigg)^{\frac{2(p-1)}{p\epsilon}}(LC_F^2)^{\frac{-(p-1)}{\epsilon}}+K_1\delta^p+
 2^{p-1}\frac{\mu }{2} LC_F^2\bigg(\frac{p}{C_p}\bigg)^{2/p}\delta^{p+\epsilon},\hspace{5mm}
\end{equation}where $K_1=\mu \max\{2^{p-2}, 1\}$. Using the stopping rule discussed in $(10)$ of Assumption $3.1$ in $(3.19)$ to obtain \begin{equation}\label{eq1} \begin{split} \gamma_{k+1}^{\delta} - \gamma_k^{\delta}  \leq \bigg[   2^{p+q-1} \beta_k^q  \frac{G_q}{q}   \frac{p}{C_p}   + \beta_k \bigg(\frac{(p-1)\epsilon_2^{\frac{p}{p-1}}}{C_p}+\frac{\epsilon_2^{-p}}{C_p}-(1+r_k)\bigg)     \bigg]\rho^2+K_1\tau^{-p}\beta_k^p\\+K_2\tau^{-\epsilon}\beta_k^{\epsilon} +K_3\tau^{-(p+\epsilon)} \beta_k^{p+\epsilon}+K_4\tau^{-1}\beta_k,\end{split}
\end{equation} where the constants $K_2, K_3$ and $K_4$ are as follows: $$K_2= 2^{p-1}\frac{\mu }{2} \bigg(\frac{C_p}{p}\bigg)^{\frac{2(p-\epsilon)}{p\epsilon}}(K_5)^{\frac{-(p-\epsilon)}{\epsilon}}, \ \ \  K_3= 2^{p-1}\frac{\mu }{2} K_5\bigg(\frac{p}{C_p}\bigg)^{2/p}, \ K_5=LC_F^2,$$ $$\text{and}\ \ K_4= K_1\bigg(\frac{C_p}{p}\bigg)^{\frac{2(p-1)}{p\epsilon}}(K_5)^{\frac{-(p-1)}{\epsilon}}.$$
Observe that under the conditions, $2^{p+q-1} \beta_k^q  \frac{G_q}{q}   \frac{p}{C_p}   + \beta_k \big(\frac{(p-1)\epsilon_2^{\frac{p}{p-1}}}{C_p}+\frac{\epsilon_2^{-p}}{C_p}-(1+r_k)\big)<0$,  $\beta_{\max}$   sufficiently small and $\tau$  sufficiently large, right side of $(3.20)$ can be less than $0$ (see Remark $3.5$). Since $\tau$ can be taken arbitrary large, for right side of $(3.20)$ to be negative, we must have
 \begin{equation}\label{eq1}  \beta_k^{q-1}     <\frac{1}{2^{p+q-1}} \frac{C_p}{p}\frac{q}{G_q}\bigg(1+r_k-\frac{(p-1)\epsilon_2^{\frac{p}{p-1}}}{C_p}-\frac{\epsilon_2^{-p}}{C_p}\bigg).  \end{equation}
Therefore, by taking  $\beta_k$'s sufficiently smaller than the one's satisfying $(3.21)$, we get
  $$ \gamma_{k+1}^{\delta} - \gamma_k^{\delta}  \leq 0 \implies \gamma_{k+1}^{\delta} \leq \gamma_k^{\delta} \leq \rho^2 \implies  u_{k+1}^{\delta} \in B.$$\end{proof}Deducing the negativity of the right side of $(3.20)$ is not an easy task because of the  involvement of so many constants. We will analyze this condition in a better way by computing some of the constants appearing in it for   Banach spaces such as $\mathcal{L}^p(\Sigma)$ (Lebesgue integrable functions), $\ell^p$ spaces for $p\geq 2$  etc. in Remark $3.5$. In the next lemma, we obtain the convergence rates for the iterates of $(2.1)$-$(2.2)$.
\begin{lemma}
Let the assumptions of Lemma  $3.2$ hold. Then, we have the following convergence rates for the  iterates of $(2.1)$-$(2.2)$: 
$$ \Delta_p^{u_0}(u^{\dagger}, u_{k_{\star}}^{\delta})-(1-K_6)\rho^2=O(\delta^{\epsilon}), \ \text{as}\ \delta \to 0.$$
\end{lemma}
\begin{proof}
From $(3.20)$ and $(3.21)$, we get  the estimate $$\gamma_{k+1}^{\delta} \leq \gamma_{k}^{\delta}+  K_1\delta^p+K_2\delta^{\epsilon}+K_3\delta^{p+\epsilon}+K_4\delta -K_6\rho^2 $$with $K_6= -2^{p+q-1} \beta_k^q  \frac{G_q}{q}   \frac{p}{C_p}   - \beta_k \big(\frac{(p-1)\epsilon_2^{\frac{p}{p-1}}}{C_p}+\frac{\epsilon_2^{-p}}{C_p}-(1+r_k)\big) >0$.   Therefore, for $0<\epsilon\leq 1$, we have  
$$ \Delta_p^{u_0}(u^{\dagger}, u_{k_{\star}}^{\delta})-(1-K_6)\rho^2=O(\delta^{\epsilon}), \ \text{as}\ \delta \to 0.$$\end{proof}
 Till now, we have proved the results of Theorem $3.1$ for  noisy data. Now, in the coming lemmas we discuss results for non-noisy data.
\begin{lemma}
Suppose that $\delta=0$ and the assumptions of Lemma $3.2$ are satisfied. Then, all the iterates of $(2.1)$-$(2.2)$ remain in $B$ and converge to the exact solution $u^{\dagger}$. Moreover, iterates satisfy the recurrence relation  \begin{equation*}\label{eq1} \begin{split} \gamma_{k+1}   \leq- K_8\gamma_k^{\frac{2}{1+\epsilon}} + \alpha_k \gamma_k + K_{11} \beta_k,      \end{split}  
\end{equation*}
where $K_8, K_{11}$ are positive constants and $\{\alpha_k\}$ is a sequence converging to $1$.
\end{lemma}
\begin{proof}
From Lemma $3.2$ via $(3.19)$ and $(3.21)$, it is easy to see that all the iterates of $(2.1)$-$(2.2)$ remain in $B$, sequence $\{\gamma_k\}$ is  monotonically decreasing and bounded below by $0$ for $\delta=0$, where $\Delta_p^{u_0}(u^{\dagger}, u_{k})= \gamma_k$ for each $k$. This means that the limit of  the sequence $\{\gamma_k\}$ exists. Let  $\underset{k \to \infty}{\lim} \gamma_k = a$. 
We show that the  sequence $\{\gamma_k\}$ converges to $0$. Putting $(3.16)$ with $\delta=0$ into  $(3.12)$ yields \begin{equation*}\label{eq1} \begin{split} \gamma_{k+1} - \gamma_k  \leq \bigg( 2^{q-1} \frac{G_q}{q}  \mu^q \hat{L}^q - \frac{\mu}{2}\bigg)\|  F(u_k)- v\|^{p} + \bigg( 2^{p+q-2}  \frac{G_q}{q} \beta_k^q + \beta_k  \frac{(p-1)\epsilon_2^{\frac{p}{p-1}}}{p}\bigg) \| u^{\dagger}-u_0 \|^p  \\     \hspace{-20mm} + \bigg(\frac{\beta_k\epsilon_2^{-p}}{ C_p}     - (1+r_k) \beta_k + 2^{p+q-2}\beta_k^q \frac{G_q}{q} \frac{p}{C_p}   \bigg) \gamma_k. \end{split}  
\end{equation*} We rewrite the above equation as \begin{equation}\label{eq1} \begin{split} \gamma_{k+1} - \gamma_k  \leq -K_7\|  F(u_k)- v\|^{p} + \bigg( 2^{p+q-2}  \frac{G_q}{q} \beta_k^q + \beta_k  \frac{(p-1)\epsilon_2^{\frac{p}{p-1}}}{p}\bigg) \| u^{\dagger}-u_0 \|^p \hspace{15mm} \\     \hspace{-20mm} + \bigg(\frac{\beta_k\epsilon_2^{-p}}{ C_p}     - (1+r_k)\beta_k + 2^{p+q-2}\beta_k^q \frac{G_q}{q} \frac{p}{C_p}   \bigg) \gamma_k, \end{split}  
\end{equation} where $K_7 = - 2^{q-1} \frac{G_q}{q}  \mu^q \hat{L}^q + \frac{\mu}{2} > 0$  because of $(3.3)$.  Taking limit $k \to \infty$ and then incorporating $\sum_{k }\beta_k < \infty 
$ and $(3.2)$ in $(3.22)$, we get \begin{equation}\label{eq1} \begin{split} a-a  \leq -K_7 \underset{k \to \infty}{\lim} \|  F(u_k)- v\|^{p} + 0 \leq -\frac{K_7}{(C_F)^{\frac{2p}{1+\epsilon}}} \underset{k \to \infty}{\lim} \gamma_k^{\frac{2}{1+\epsilon}} = -K_8 \underset{k \to \infty}{\lim} \gamma_k^{\frac{2}{1+\epsilon}}, \end{split}  
\end{equation} where $K_8 = \frac{K_7}{(C_F)^{\frac{2p}{1+\epsilon}}}$ is a  positive constant. Now, using the continuity of the function $x \to x^{a}$ for any $a > 1$, $(3.23)$ implies that \begin{equation*}\label{eq1} \begin{split} 0  \leq -K_8a^{\frac{2}{1+\epsilon}} \implies a^{\frac{2}{1+\epsilon}} \leq 0.  \end{split}  
\end{equation*} But as $\gamma_k \geq 0$,  we must have $a \geq 0$ and thus above implies that $a = 0$. Hence, by $(4)$ in Theorem $2.2$, $u_k \to u^{\dagger}$, i.e. iterates of $(2.1)$-$(2.2)$ converges to the exact solution for non-noisy data. 
Next, we find  the recursion formula satisfied by the sequence   $\{\gamma_k\}$. 
Using $(3.2)$ and $u_0 \in B$ in $(3.22)$ to reach at \begin{equation*}\label{eq1} \begin{split} \gamma_{k+1}  \leq -K_8\gamma_k^{\frac{2}{1+\epsilon}} + \bigg( 2^{p+q-2}  \frac{G_q}{q} \beta_k^q + \beta_k  \frac{(p-1)\epsilon_2^{\frac{p}{p-1}}}{p}\bigg) \frac{p}{C_p}\rho^2+\hspace{40mm}  \\      \bigg(1+ \frac{\beta_k\epsilon_2^{-p}}{ C_p}     - (1+r_k)\beta_k + 2^{p+q-2}\beta_k^q \frac{G_q}{q} \frac{p}{C_p}   \bigg) \gamma_k \end{split}  
\end{equation*} \begin{equation}\label{eq1} \begin{split} = -K_8\gamma_k^{\frac{2}{1+\epsilon}} + \alpha_k \gamma_k + K_9 \beta_k^q + K_{10}\beta_k  \leq- K_8\gamma_k^{\frac{2}{1+\epsilon}} + \alpha_k \gamma_k + K_{11} \beta_k,  \end{split}  
\end{equation} where constants $K_9, K_{10}$ and $\alpha_k$ are as follows: $$K_9 = 2^{p+q-2}  \frac{G_q}{q} \frac{p}{C_p}\rho^2,\quad K_{10} =  \frac{(p-1)\epsilon_2^{\frac{p}{p-1}}}{C_p}\rho^2, \quad \alpha_k = 1+ \frac{\beta_k\epsilon_2^{-p}}{ C_p}     - \beta_k(1+r_k) + 2^{p+q-2}\beta_k^q \frac{G_q}{q} \frac{p}{C_p},$$ and the last term in $(3.24)$ is written because $\beta_k < 1$ and $q>1$,  where $K_{11} = K_9+K_{10}$.  We can easily see that $\alpha_k \to 1$. So, $(3.24)$ is the required recurrence relation. 
\end{proof}
In the next lemma, we obtain the convergence rates for noise-free iterates in terms of radius $\rho$ of $B$.
\begin{lemma}
Let the assumptions of Lemma $3.4$ hold and there exists  a constant $C\geq 0$ such that $\beta_k\leq C\gamma_k$ for each $k$. Then, for $\epsilon\in (0, 1)$ we have the following convergence rate \begin{equation*}\label{eq1} \begin{split}\Delta_p^{u_0}(u^{\dagger}, u_k) \leq   \bigg(\big( g_k \rho^2\big)^{-\frac{1-\epsilon}{1+\epsilon}}  + h_k \bigg)^{-\frac{1+\epsilon}{1-\epsilon}}, \quad k =  1, 2, \ldots    \end{split}  
\end{equation*}
For $\epsilon=1$, we have \begin{equation*}\label{eq1} \begin{split} \Delta_p^{u_0}(u^{\dagger}, u_k)    \leq \prod_{i=0}^{k-1}(- K_8 + \alpha_i + K_{11}C) \rho^2, \quad k =   1, 2,  \ldots    \end{split}  
\end{equation*} $($see proof for the meaning of constants $g_k, h_k$ for $k\geq 1)$.
\end{lemma}

\begin{proof}
With the given condition $\beta_k \leq C \gamma_k$,  $(3.24)$ can be written as  \begin{equation}\label{eq1} \begin{split} \gamma_{k+1}  \leq   -K_8\gamma_k^{\frac{2}{1+\epsilon}} + \alpha_k \gamma_k +  K_{11} \beta_k  \leq- K_8\gamma_k^{\frac{2}{1+\epsilon}} + d_k \gamma_k \\ = d_k \gamma_k\bigg( 1- e_k\gamma_k^{\frac{1-\epsilon}{1+\epsilon}}  \bigg),\end{split}  
\end{equation} where  $d_k = \alpha_k +CK_{11}$ and $e_k = \frac{K_8}{d_k}$ for every $k$. Let $t = \frac{1-\epsilon}{1+\epsilon}$.  Then,  $(3.25)$ yields \begin{equation}\label{eq1} \begin{split} (\gamma_{k+1})^{-t}  \geq   ( d_k \gamma_k)^{-t} \big( 1- e_k\gamma_k^{t}  \big)^{-t}. \end{split}  
\end{equation}  
Applying the estimate $(1-y)^{-t} \geq 1+ ty, \ \forall   y \in (0, 1) $ into $(3.26)$ for $k \geq 0$, we get\begin{equation*}\label{eq1} \begin{split} (\gamma_{k+1})^{-t}  \geq   ( d_k \gamma_k)^{-t}  + f_k,    \end{split}  
\end{equation*} where $f_k= t e_k d_k ^{-t}$. Thus, we have\begin{equation*}\label{eq1} \begin{split}\Delta_p^{u_0}(u^{\dagger}, u_k) \leq   \bigg(\big( g_k \rho^2\big)^{-\frac{1-\epsilon}{1+\epsilon}}  + h_k \bigg)^{-\frac{1+\epsilon}{1-\epsilon}}, \quad k =  1, 2, \ldots    \end{split}  
\end{equation*} where \begin{equation*}
g_k = \prod_{i=0}^{k-1}d_i, \ \text{for} \  k \geq 1,\end{equation*} and \begin{equation*} h_k = \sum_{j=1}^{k-1}\bigg(d_jd_{j+1}\ldots d_{k-1}\bigg)^{-\frac{1+\epsilon}{1-\epsilon}}f_{j-1} + f_{k-1}, \ k \geq 2, \  h_1 = f_0. 
\end{equation*}
For $\epsilon = 1$, $(3.25)$ with $\beta_k \leq C\gamma_k$ implies that
\begin{equation*}\label{eq1} \begin{split} \gamma_{k}    \leq \prod_{i=0}^{k-1}(- K_8 + \alpha_i + K_{11}C) \rho^2, \quad k =   1, 2,  \ldots    \end{split}  
\end{equation*}
So, we get the convergence rates via in terms of radius $\rho$ of $B$.
\end{proof} 
\begin{remark}
 The condition $\beta_k \leq C\gamma_k$ assumed in Lemma $3.5$ is an abstract smoothness condition for obtaining the convergence rates and is similar to other smoothness concepts $($e.g.  source conditions, variational inequalities$)$ already available in the literature $[4, 16, 26, 27, 33]$ in the sense that all these incorporate some a-priori knowledge  of the exact solution.
\end{remark}
For the noise free iterates of $(2.1)$-$(2.2)$, we also obtain convergence rates in terms of $\beta_k$'s where $\beta_k$ satisfy $(2.2)$ for each $k$.
\begin{lemma}
In addition to the rates obtained in Lemma $3.5$, we  also obtain the rates 
 \begin{equation*} \Delta_p^{u_0}(u^{\dagger}, u_k) =  O(\beta_k^{q-1}) \quad \text{as} \  k \to \infty,
\end{equation*} provided the assumptions of Lemma $3.4$ hold and \begin{equation*}     K_{12} +  \eta \beta_k^{-1}\bigg[ \alpha_k      -  \bigg( \frac{\beta_{k+1}}{\beta_{k}}\bigg)^{q-1}\bigg] \leq 0, 
\end{equation*} where $\eta, K_{12}$ are positive constants and $\{\alpha_k\}$ is a sequence converging to $1$.
\end{lemma}
\begin{proof}
 The inequality $(3.24)$ leads to the estimate \begin{equation}  \gamma_{k+1}  \leq   -K_8\gamma_k^{\frac{2}{1+\epsilon}} + \alpha_k \gamma_k +  K_{12} \beta_k^{q}, 
\end{equation} where $K_{12}$ is such that $ K_{12} \beta_k^{q} > K_{11}\beta_k$ for every $k$ (such a condition is possible as $\beta_{\max}<\infty$).
Now, let us define $\eta_k = \dfrac{\gamma_k}{\beta_k^{q-1}}$. Then from $(3.27)$, we have \begin{equation*}\eta_{k+1} \leq  \bigg( \frac{\beta_k}{\beta_{k+1}}\bigg)^{q-1} \bigg[-K_8 \eta_k^{1+t} \beta_k^{(q-1)t}  + \alpha_k \eta_k  + K_{12}  \beta_k    \bigg] \leq \bigg( \frac{\beta_k}{\beta_{k+1}}\bigg)^{q-1} \big[ \alpha_k \eta_k  + K_{12}  \beta_k    \big],  
\end{equation*} where $t = \frac{1-\epsilon}{1+\epsilon}$. For the uniform boundedness of $\{\eta_k\}$ by some $\eta$, sufficient condition is    \begin{equation*}  \alpha_k \eta  + K_{12} \beta_k    \leq \eta \bigg( \frac{\beta_{k+1}}{\beta_{k}}\bigg)^{q-1}\implies     K_{12} +  \eta \beta_k^{-1}\bigg[ \alpha_k      -  \bigg( \frac{\beta_{k+1}}{\beta_{k}}\bigg)^{q-1}\bigg] \leq 0.   \end{equation*}   Thus, we have $ \Delta_p^{u_0}(u^{\dagger}, u_k) = \gamma_k =  O(\beta_k^{q-1}) \  \text{as} \  k \to \infty.$\end{proof}
 On combining Lemmata $3.1$-$3.6$, one can see that proof of the Theorem $3.1$ is complete. Observe that, for $0<\epsilon<1$, in the case of non-noisy data we have obtained the sub-linear convergence rates in  Lemma $3.5$  and as $\epsilon\to 1$, speed of the convergence increases because it switches  to the linear convergence. Further, in Lemma $3.6$ we have obtained the rates in terms of $\beta_i$'s and rates are sub-linear or super-linear accordingly as $1<q<2$ or $q>2$ respectively. 

For proving Theorem $3.1$ (especially Lemma $3.2$), we require the condition $2^{p+q-1} \beta_k^q  \frac{G_q}{q}   \frac{p}{C_p}   + \beta_k \big(\frac{(p-1)\epsilon_2^{\frac{p}{p-1}}}{C_p}+\frac{\epsilon_2^{-p}}{C_p}-(1+r_k)\big)<0$. We discuss about this condition in the following remark.
\begin{remark}
For $(3.21)$ to be satisfied, we must have the following: \begin{enumerate}[(i)]
\item $2^{p+q-1} \beta_k^{q-1}  \frac{G_q}{q}   \frac{p}{C_p}   < 1+r_k -\frac{(p-1)\epsilon_2^{\frac{p}{p-1}}}{C_p}-\frac{\epsilon_2^{-p}}{C_p}$, for each $k$.
\item $\frac{(p-1)\epsilon_2^{\frac{p}{p-1}}}{C_p}+\frac{\epsilon_2^{-p}}{C_p}<1+r_k$, for each $k$.
\end{enumerate}
Observe that $(i)$ can be easily handled by taking $\beta_{\max}$ sufficiently small provided $(ii)$ holds. For $(ii)$, first, we find the values of $C_p$ for different $p$.
Using $(2.7)$, we get the following table $($values of $K_p$  can be obtained by writing a simple program in MATLAB, C${++}$ etc.$)$:\vspace{2mm} \begin{center}
    \begin{tabular}{ | l | l | }
    \hline
   Value of $p$ & Value of $K_p$ $($approx.$)$  \\ \hline
    $1.5$ & $3.8132$  \\ \hline
    $2$ & $5.2086$  \\ \hline
     $3$ & $7.3326$  \\ \hline
    $4$ & $8.9576$  \\
    \hline  $5$ & $10.2274$\\ \hline  $\cdots$ & $\cdots$\\ \hline  $10$ & $13.4980$\\ \hline
    \end{tabular}\vspace{0mm}
    \captionof{table}{Relationship between $p$ and $K_p$}\label{Relationship between $p$ and $K_p$}
\end{center}
From Remark $2.4$, we know that  $C_p=\frac{YK_p}{2^p}$, where the constant $Y$ is same as appearing in the Definition $2.4$. Now, if $U=\mathcal{L}^p(\Sigma)$,  where $p\geq 2$ and $\Sigma\subset \mathbb{R}^n$ is an open domain, then from Example $2.2$ we know that $\delta_U(\epsilon)=\epsilon^p$ for any $\epsilon\in (0, 2]$. In other words, $U$ is $p$ convex for any $p\geq 2$ with $Y=1$. Therefore,  $$C_p=\frac{K_p}{2^p},\quad\text{for}\ \ U=\mathcal{L}^p(\Sigma).$$
So, $(ii)$ holds provided $$(p-1)\epsilon_2^{\frac{p}{p-1}}+ \epsilon_2^{-p}<\frac{(1+r_k)K_p}{2^p}.$$
For instance, take $p=2$ and $\epsilon_2=1$. Then the last inequality becomes $$ \epsilon_2^{2}+ \epsilon_2^{-2}= 2<\frac{(1+1)(5.2086)}{4},$$ which is true since $r_k=1$ for each $k$ $($see $(6)$ of Assumption $3.1)$. This means our assumption of Lemma $3.2$ related to $\beta_k$ is satisfied. 
 The  rationale behind introducing $\epsilon_2$ in $(3.10)$ is  to make the estimate $(ii)$ more flexible so that it  holds for a range of values of $r_k$ or $\epsilon_2$. In other words one can see that $(ii)$ also holds for $p=2$ and $\epsilon_2=0.9$. Further, from $[38, \ \text{Equation}\ 1.2]$ we can see that  the Banach spaces $\ell^p, W_m^p$ $($Sobolev space$)$ are $\max\{2, p\}$ convex and one can verify the condition $(ii)$   provided exact  bound for $Y$ is known as in the case of $\mathcal{L}^p$ spaces. 
\end{remark}

 Next result is for the crucial case when $\epsilon = 0$ in $(3.2)$.  We need to do this case separately as we can not take $\epsilon=0$ in the proof of Theorem $3.1$ (see $(3.16)$, which contains a term having $\epsilon$ in the denominator).  For the case $\epsilon=0$, we need to have a different bound on $\mu$ than what is assumed in Theorem $3.1$ ($(8)$ of Assumption $3.1$).
\begin{theorem}
Let $F$ be a non-linear operator between the Banach spaces $U$, $V$ and the operator equation $F(u) = v, v \in V$, has  a solution $u^{\dagger}$. Let the conditions $(1)$-$(7)$, $(10)$  of Assumption $3.1$ hold with $\epsilon = 0$ in $(5)$, and  $\mu$ satisfies\begin{equation} \mu^{q-1} < 
\frac{q}{2^{q-1}G_q \hat{L}^q} \bigg[ 1-  \frac{1}{2} LC_F^2\bigg(\frac{p}{C_p}\bigg)^{2/p} \bigg].  
\end{equation}  Suppose that  $v^{\delta}\in V$ be such that $\|v^{\delta}-v\|\leq \delta$. 
 Then all the iterates $u_{k+1}^{\delta}$  of iteratively regularized Landweber iteration  method $(2.1)$-$(2.2)$ remain in $B$ for all $k\leq k_{\star}(\delta)-1$ provided  $(3.21)$ holds. Moreover, iterates satisfy the following recurrence relation
$$\gamma_{k+1}^{\delta} \leq \gamma_k^{\delta}+ M_1\delta+M_2\delta^{p}-K_6\rho^2,$$  and we obtain the following convergence rate $$\ \Delta_p^{u_0}(u^{\dagger}, u_{k_{\star}})-(1-K_6)\rho^2=O(\delta), \ \text{as}\ \delta \to 0,$$  for some constants $M_i, 1\leq i\leq 2$, where $\gamma_k = \Delta_p^{u_0}(u^{\dagger}, u_k^{\delta})$ and constant $K_6$ has same meaning as in Theorem $3.1$. 
\end{theorem}

\begin{proof}
It can be observed  that Lemma $3.1$ is valid with our assumptions. So, put $\epsilon = 0$ in  $(3.12)$  and employ the condition $(3.28)$ in it to obtain
\begin{equation}\label{eq1} \begin{split}   \gamma_{k+1}^{\delta} - \gamma_k^{\delta}
    \leq \mu \delta\|F(u_k^{\delta})-v^{\delta}\|^{p-1}   + \bigg( 2^{p+q-2}  \frac{G_q}{q} \beta_k^q + \beta_k  \frac{(p-1)\epsilon_2^{\frac{p}{p-1}}}{p}\bigg) \| u^{\dagger}-u_0 \|^p   \\ +   \bigg(\frac{\beta_k\epsilon_2^{-p}}{ C_p}     - (1+r_k)\beta_k + 2^{p+q-2}\beta_k^q \frac{G_q}{q} \frac{p}{C_p}   \bigg) \gamma_k^{\delta}. \end{split}  
\end{equation} 
Let us assume that $u_k^{\delta}\in B$, then $(3.29)$ with $(2.3)$ leads to 
\begin{equation}\label{eq1} \begin{split}   \gamma_{k+1}^{\delta} - \gamma_k^{\delta}
    \leq \mu \delta\|F(u_k^{\delta})-v^{\delta}\|^{p-1}   + \bigg[ 2^{p+q-1} \beta_k^q \frac{G_q}{q}\frac{p}{C_p}    + \beta_k\bigg( \frac{(p-1)\epsilon_2^{\frac{p}{p-1}}}{C_p}+\frac{\epsilon_2^{-p}}{ C_p}-(1+r_k) \bigg)\bigg]\rho^2. \end{split}  
\end{equation} 
From $(3.14)$ we know that \begin{equation*}  \|F(u_k^{\delta})-v^{\delta}\|  \leq \hat{L} \bigg(\frac{p}{C_p}\bigg)^{1/p}\rho^{\frac{2}{p}}+{\delta}.
\end{equation*}As $p-1$ is either $\leq 1$ or $>1$,  above with the estimate $$(r_1+s_1)^{\lambda}\leq 2^{\lambda-1}(r_1^{\lambda}+s_1^{\lambda})\ \text{for}\ r_1, s_1\geq 0\ \text{and}\ \lambda\geq 1,$$ or $$(r_1+s_1)^{\lambda}\leq r_1^{\lambda}+s_1^{\lambda}\ \text{for}\ r_1, s_1\geq 0\ \text{and}\ 0\leq \lambda\leq 1,$$ accordingly as $p-1\leq 1$ or $p-1\geq 1$ with $(3.30)$ imply that  \begin{equation*}\label{eq1} \begin{split} \gamma_{k+1}^{\delta} - \gamma_k^{\delta}  \leq  \bigg[ 2^{p+q-1} \beta_k^q \frac{G_q}{q}\frac{p}{C_p}    + \beta_k\bigg(\frac{(p-1)\epsilon_2^{\frac{p}{p-1}}}{C_p}+ \frac{\epsilon_2^{-p}}{ C_p}-(1+r_k) \bigg)\bigg]\rho^2+K_1\mu\delta^p\\  +K_1\mu \delta \hat{L}^{p-1}\bigg(\frac{p}{C_p}\bigg)^{\frac{p-1}{p}}\rho^{\frac{2(p-1)}{p}}, \end{split} \end{equation*}where $K_1=\max\{1, 2^{p-2}\}$. Use  the stopping rule discussed in $(10)$ of Assumption $3.1$ in above to obtain 
\begin{equation*}\label{eq1} \begin{split} \gamma_{k+1}^{\delta} - \gamma_k^{\delta}  \leq  \bigg[ 2^{p+q-1} \beta_k^q \frac{G_q}{q}\frac{p}{C_p}    + \beta_k\bigg( \frac{(p-1)\epsilon_2^{\frac{p}{p-1}}}{C_p}+\frac{\epsilon_2^{-p}}{ C_p}-(1+r_k) \bigg)\bigg]\rho^2  +M_1\delta+ M_2\delta^p\hspace{10mm} \end{split} \end{equation*} \begin{equation*}\label{eq1} \begin{split}  \leq  \bigg[ 2^{p+q-1} \beta_k^q \frac{G_q}{q}\frac{p}{C_p}    + \beta_k\bigg( \frac{(p-1)\epsilon_2^{\frac{p}{p-1}}}{C_p}+ \frac{\epsilon_2^{-p}}{ C_p}-(1+r_k) \bigg)\bigg]\rho^2   +M_1 \beta_k\tau^{-1}+ M_2\beta_k^p\tau^{-p}, \end{split} \end{equation*} where $M_1= K_1\mu  \hat{L}^{p-1}\big(\frac{p}{C_p}\big)^{\frac{p-1}{p}}\rho^{\frac{2(p-1)}{p}}$ and $M_2= K_1\mu$. Now if  $(3.21)$ holds and $\tau$ is sufficiently large, then we can see that $u_{k+1}^{\delta}\in B$ as in Lemma $3.2$. Rest part of the proof follows on the lines of Lemma $3.3$.
\end{proof}
\begin{remark}
In Theorem $3.2$, we have obtained the convergence rates only for noisy data. However, convergence  rates can also be obtained for non-noisy data, in the case when $\epsilon=0$ in $(3.2)$, exactly on the lines of Lemmata $3.4$-$3.6$ as in Theorem $3.1$.
\end{remark}
In the following remark, we discuss about the special case  when $\{\beta_k\}=\{0\}$. Observe that in this case, $(2.1)$-$(2.2)$ reduces to Landweber iteration method (non-noisy version) as discussed in Remark $2.3$.
\begin{remark}
\begin{enumerate}[(i)]
\item If $\beta_k=0$ for each $k$, then one can reproduce Lemmata $3.1$, $3.2$, $3.3$ to obtain the convergence rates for Landweber iteration method in the case of noisy data which are missing from $[24]$.
\item If $\beta_k=0$ for each $k$ and $\delta=0$, then observe that the right hand side of $(3.19)$ is trivially satisfied and $\alpha_k=0$ $($see Lemma $3.4)$. Also, take $C=0$ in Lemma $3.5$ which means $d_k=1$, $e_k=K_8$ and $f_k=tK_8$ for each $k$. Therefore, one can see that the rates obtained in Lemma $3.5$ and the rates obtained in $[24]$  for Landweber iteration method are of the same order. Hence, we can say that the results of $[24]$  can be deduced from our results in a special case.

\end{enumerate}
\end{remark}

\section{Electrical Impedance Tomography (EIT)}In this section, we discuss an example  related to  Calder\'on's inverse problem which satisfies the H\"older stability estimate $(3.2)$ under some assumptions on the electrical conductivity.  
This problem has been also considered in $[24]$ to obtain the convergence and convergence rates of non-linear Landweber iteration scheme in Banach spaces.  
 Our results on covergence can be applied  on the Calder\'on's inverse problem which is the mathematical bedrock of  EIT.     It is well known that this Calder\'on's inverse problem is severely ill-posed $[11]$. Ulhmann, in $[14]$ has recently studied the EIT and Calder\'on's problem and further, we refer to $[19, 20, 28, 29]$  for some of the literature in this context. In $[6, 11]$, two results  on the Lipschitz-type stability estimates have been obtained for the Calder\'on's inverse conductivity problem provided  the a-priori information about the conductivity is known, i.e. it is piecewise constant with a bounded number of unknown values. The difference between these two results is that, in $[11]$  a real valued case is discussed whereas a complex valued case is discussed in $[6]$.   In our work, we consider the real valued case which involves the  determination of $v \in H^1(\Omega)$ where $v$ satisfies \begin{equation} 
 \begin{cases} \text{div}(\gamma \nabla v) = 0, \ \  \text{in} \  \Omega \\ \ \   v = g, \qquad   \ \quad \text{on} \ \partial \Omega.\end{cases} \end{equation} Here $g \in H^{1/2}(\partial \Omega)$,  $\Omega \subset \mathbb{R}^n, n \geq 2$ is a bounded domain having smooth boundary and $\gamma$ is the positive and bounded function representing the electrical conductivity of $\Omega$. If $\gamma$ is a complex valued function in $(4.1)$, then on subjecting to Dirichlet boundary conditions, $(4.1)$ also appears as the asymptotic limit of an
elliptic equation with memory in the study of electrical conduction in biological tissues $[6]$. Further, Calder\'on's inverse  problem has many applications, for instance, in the fields of nondestructive testing of materials, medical imaging, and therefore, it is an important problem to study.

  The inverse problem associated with 
EIT is the determination of electrical conductivity $\gamma$  from the information of $\Lambda_{\gamma}$, i.e.  the Dirichlet to Neumann map which is defined as $$\Lambda_{\gamma}: H^{1/2}(\partial \Omega) \to H^{-1/2}(\partial \Omega): \ g \to \bigg(\gamma \frac{\partial v}{\partial \nu}\bigg)\bigg|_{\partial \Omega},$$ where the vector $\nu$ is the outward normal to $\partial \Omega$. The  operator $F$ associated with the inverse problem is defined by \begin{equation}
 F : U \subset \mathcal{L}_{+}^{\infty}(\Omega) \to \mathbb{L}(H^{1/2}(\partial \Omega), H^{-1/2}(\partial \Omega)): F(\gamma) = \Lambda_{\gamma},\end{equation} where $\mathbb{L}(H^{1/2}(\partial \Omega), H^{-1/2}(\partial \Omega))$ is the space of all bounded linear operators from $H^{1/2}(\partial \Omega)$ to  $H^{-1/2}(\partial \Omega)$.
 Further,   $F'$, the Fr\'echet derivative of $F$  at $\gamma = \gamma'$ is given by \begin{equation*}
 F'(\gamma'): U \subset \mathcal{L}^{\infty}(\Omega) \to \mathbb{L}(H^{1/2}(\partial \Omega), H^{-1/2}(\partial \Omega)): \delta \gamma \to \ F'(\gamma')(\delta \gamma), 
 \end{equation*} where  $F'(\gamma')(\delta \gamma)$ is defined by the sesquilinear form \begin{equation*}
 \langle F'(\gamma')(\delta \gamma) g_1, \ g_2  \rangle = \int_{\Omega} \delta \gamma \nabla v_1 \cdot \nabla v_2 dx, \quad g_1, g_2 \in H^{1/2}(\partial \Omega),\end{equation*} where $v_1$ and $v_2$ are the weak solutions of $$\begin{cases}
 \text{div} (\gamma' \nabla v_1) = 0 = \text{div}  (\gamma' \nabla v_2), \quad \text{in} \  \Omega \\ v_1 = g_1, \ \ v_2 = g_2 \qquad \quad \quad \qquad \text{on} \ \partial \Omega. \end{cases}$$ Under the assumption that $\gamma \in L^{\infty}(\Omega)$, for the case $n = 2$, uniqueness of the solution to the inverse problem $(4.2)$ is discussed in $[20]$ and for $n \geq 3$, it is considered in $[22]$ provided $\gamma$ is in the Sobolev space $W^{3/2, \infty}(\Omega)$.

Remaining discussion of this section is mainly based on the results of $[11]$. So, we refer this article whenever needed instead of recalling all the results. Next theorem presents the  Lipschitz estimate  established in $[11]$.
\begin{theorem}
Let $\gamma_1, \gamma_2$ be two real piecewise constant  functions  such that $$\gamma_i(x)=\sum_{j=1}^N \gamma_j^i(x)\chi_{D_j}(x),\ x\in \Omega,\ \lambda\leq \gamma_i(x)\leq \lambda,\ i=1, 2,$$ where $\lambda\in (0, 1]$, $\gamma_j^i$ is an   unknown real number for each $i, j$, $D_j$'s are known open sets, $\chi_{D_j}$ is characteristics function of the set $D_j$  and $N\in \mathbb{N}$. Then under some assumptions on $\Omega$, $D_j$'s $($see section $2.2$ in $[11])$,  we have \begin{equation*}
\|\gamma_1-\gamma_2\|_{L^{\infty}(\Omega)} \leq C \|\Lambda_{\gamma_1}-\Lambda_{\gamma_2}\|_{L(H^{1/2}(\partial \Omega), H^{-1/2}(\partial \Omega))},
\end{equation*} where $C$ is a constant. 
\end{theorem}

Now, we verify that the assumptions of Theorem $3.1$ are satisfied. Before that, observe that the Banach space $L^{\infty}(\Omega)$ is not a uniformly convex space, so defining the space $U$ in accodance with  Theorem $4.1$ as\begin{equation*}
 U= \text{span} \{\chi_{D_1}, \chi_{D_2}, \ldots, \chi_{D_N}\}\end{equation*} fitted with $L^p$ norm where $p > 1$, $D_i$'s, $\chi_{D_i}'s$ are same as in Theorem $4.1$. Then,  with the help of basis $\{\chi_{D_1}, \chi_{D_2}, \ldots, \chi_{D_N}\}$, one can show the Lipschitz continuity of  $F'$ and its boundedness $[24, \ \text{subsection}\ 5.3]$, i.e.  $ (2), (3)$  in Assumption $3.1$.

Further, assume that $v = F(\gamma^{\dagger})$ where $\gamma^{\dagger} \in U$. Then $(5)$ of Assumption $3.1$ holds (see Theorem $4.1$). As the notion of weak and strong topology is equivalent for the finite dimensional spaces, $F$ defined in $(4.2)$ is weakly sequentially closed which means $(4)$ of Assumption $3.1$ holds. And let $u_0$, $\{\beta_k\}$, $\mu$, $\rho^2$ and $k_*$ are chosen in accordance with Theorem $3.1$,  then iteratively regularized  Landweber iteration method converges in accordance with  Theorem $3.1$  and we also get   the said convergence rates. 
\begin{remark}
It is of possible impression that the inverse problem $(4.2)$ becomes well posed by considering the unknown conductivities in a finite dimensional space in Theorem $4.1$. However, a counter example to discomfort such an impression is  discussed in $[11]$. We recall that example in our work for the sake of completeness. Let $F:\mathbb{R}\to\mathbb{R}^3$ be such that $$F(t)= \big((2+\cos 3\pi \alpha t)\cos  2\pi t, (2+\cos 3\pi \alpha t)\sin  2\pi t, \sin  2\pi\alpha t\big), \ t\in \mathbb{R},$$ where $\alpha$ is a parameter. It can be shown that $F$   is smoothly locally invertible. For $\alpha$ rational, $F$ is periodic whereas for $\alpha$ irrational, it is globally one-to-one but $F^{-1}$ is discontinuous at every point. Further, if $\alpha$ is irrational and $F$ is restricted to interval $[-I, I]$ for some $I>0$, then $F^{-1}$ is globally Lipschitz, but  Lipschitz constant may blow up as $\alpha$ tends to any rational number.
\end{remark}
\section{Conclusion}
We have implemented the iteratively regularized  Landweber iteration scheme  for  non-linear inverse problems in Banach spaces to obtain the  convergence rates. Under the condition that non-linear operator satisfies H\"older stability estimate, we proved the convergence for non-noisy iterates and obtained the   sublinear convergence rates   under some additional assumptions. To the best of  our knowledge, this paper is the  first advancement to find the explicit reconstructions for Iteratively regularized landweber iteration method by employing  the  H\"older stability estimates  after the reconstructions in $[24]$. An important thing to note is  that while obtaining the convergence rates for iteratively regularized Landweber iteration method $(2.1)$-$(2.2)$, complementary, we also get the convergence rates for Landweber scheme (see Remark $3.7$)  in the presence of noisy data which are missing from the literature.

An important future work in this direction is to come up with a situation where the assumptions of Theorem $3.1$ are satisfied in an infinite dimensional Banach space, for instance, one can think of fitting $L^{\infty}(\Omega)$ in some infinite dimensional space in Example $4.1$.
  Further, one can also think about the application of  results of Theorem $3.1$ in option pricing theory $($OPT$)$, see e.g. $[37]$.
The inverse problem associated with OPT is addressed in $[32]$.

\bibliographystyle{plain}

\end{document}